\documentclass[letterpaper,10pt,3p,preprint]{elsarticle}

\usepackage{amsmath,amssymb,amsfonts}
\usepackage{bm}

\usepackage{xcolor}
\usepackage{graphicx}
\usepackage{subcaption}
\usepackage{comment}
\usepackage{soul}

\usepackage[hidelinks]{hyperref}

\numberwithin{equation}{section}
\newtheorem{theorem}{Theorem}
\newproof{proof}{Proof}

\newcommand{\Nbb}{\mathbb{N}}
\newcommand{\Rbb}{\mathbb{R}}
\newcommand{\Vcal}{\mathcal{V}}
\newcommand{\Ecal}{\mathcal{E}}
\newcommand{\vect}[1]{\bm{#1}}
\newcommand{\matr}[1]{\bm{#1}}
\DeclareMathOperator*{\argmin}{arg\,min}

\begin{document}

\begin{frontmatter}

\title{Machine-precision energy conservative reduced models
    for Lagrangian hydrodynamics by quadrature methods}

\author[1]{Chris Vales}
\ead{chris.vales@dartmouth.edu}
\author[2]{Siu Wun Cheung\corref{cor1}}
\ead{cheung26@llnl.gov}
\author[2]{Dylan M. Copeland}
\ead{copeland11@llnl.gov}
\author[2]{Youngsoo Choi}
\ead{choi15@llnl.gov}

\cortext[cor1]{Corresponding author.}
\affiliation[1]{organization={Department of Mathematics,
    Dartmouth College},
    city={Hanover},
    state={NH 03755},
    country={USA}}
\affiliation[2]{organization={Center for Applied Scientific
    Computing, Lawrence Livermore National Laboratory},
    city={Livermore},
    state={CA 94550},
    country={USA}}

\begin{abstract}
We present an energy conservative, quadrature based model reduction
framework for the compressible Euler equations of Lagrangian
hydrodynamics.
Building on a finite element discretization of the governing
equations, we develop reduced models using data based reduced
basis functions and the empirical quadrature procedure (EQP).
We introduce a strongly energy conservative variant of EQP that
enforces exact energy conservation in the reduction process.
Numerical experiments for four benchmark problems---Sedov blast,
Gresho vortex, triple point and Taylor–Green vortex---demonstrate
that the numerical implementation of our proposed method
conserves total energy to near machine precision, while
maintaining accuracy comparable to the basic EQP formulation.
\end{abstract}


\end{frontmatter}

\section{Introduction}\label{sec:intro}
High fidelity simulations of multiscale, multiphysics systems
often involve a large number of degrees of freedom evolving over widely
separated temporal and spatial scales, making them prohibitively
expensive for tasks such as design, control or uncertainty
quantification \cite{Benner2015,Peherstorfer2018}.
Model reduction methods---and the related closure or parametrization
methods---aim to construct low dimensional surrogate models that
retain the essential dynamics of the original system, while
reducing the computational cost associated with its numerical
simulation \cite{Sagaut2006,E2011,Stensrud2013}.
Reduced models can be derived using both physical modeling and
data based approaches, depending on the complexity of the
underlying dynamics and the availability of high quality
data, among other determining factors.
In either case, reduction strategies employ a wide variety of
mathematical tools, including both deterministic and
stochastic methods.

In this work we derive reduced models that govern the evolution
of the linear projection of the original state variables onto
a subspace spanned by a set of reduced basis functions.
Focusing on Lagrangian hydrodynamics, our approach combines
physical knowledge about the underlying system in the form of its
governing equations, with basis functions computed in a data
based fashion.

In the remainder of Section \ref{sec:intro}
we discuss projection based reduction methods and
modeling approaches for Lagrangian hydrodynamics,
followed by a summary of the main results of this work.
In Section \ref{sec:fom} we present the governing equations and
the discretization used to obtain the full model.
The reduced model is derived in Section \ref{sec:rom},
followed by a presentation of the EQP hyper-reduction method
and its implementation.
The energy conservative EQP method is developed in
Section \ref{sec:ceqp}.
Section \ref{sec:experiments} presents our numerical
results for four different problem cases using both EQP methods,
followed by a conclusion in Section \ref{sec:conclusion}.
Finally, \ref{app:implementation}
includes additional details on the numerical
implementation of the EQP methods.

\subsection{Projection based reduction}
Projection based reduction methods derive reduced models by
replacing the original state space of the dynamics with a
subspace defined as the linear span of a set of reduced
basis functions \cite{Benner2015,Swischuk2019}.
The reduced state variables are then defined by the linear
projection of the original state variables to the constructed
subspace.
The reduced basis functions that define the state space
of the reduced dynamics can be computed by various data
based methods, which usually employ the singular value
decomposition (SVD) or eigenvalue decomposition
of an appropriate data matrix.

The most common method is arguably the proper orthogonal
decomposition (POD), which uses the SVD of a snapshot data
matrix to construct an orthonormal basis for its range space
\cite{Aubry1991,Berkooz1993,Hinze2005}.
Given a collection of samples of the state variables,
POD yields the best linear fixed-rank approximation
to the data in $\Rbb^N$ with the euclidean inner product,
where $N\in\Nbb$ denotes the state space dimension.
An alternative is the dynamic mode decomposition (DMD)
\cite{Schmid2022},
which can be used to approximate the eigenspaces of
the system's Koopman operator restricted to the finite
dimensional subspace spanned by the sampled data
\cite{Budisic2012,Brunton2022,Colbrook2024}.

The POD and DMD methods can be viewed as specific instances
of kernel methods, where one employs a symmetric kernel
function $k$ to compute the pairwise distances matrix
$K_{ij}=k(u_i,u_j)$
from state samples $u_i$, $u_j\in\Rbb^N$,
and uses the eigenvectors of the kernel matrix $K$
as the reduced basis functions.
Compared to POD, general kernel methods replace
the correlation kernel based on the euclidean inner
product with the kernel function $k$,
thereby embedding the state samples into a
feature space before computing their pairwise distances
\cite{Schoelkopf2001}.
Kernel methods can also be used to approximate the
eigenspaces of the restricted Koopman operator by using
kernels that act on delay embedded state samples
\cite{Giannakis2019vsa}.

The effectiveness of linear subspace projection methods is generally
limited to problems where the solution manifold can be accurately
approximated by a low dimensional subspace.
This assumption is not satisfied in problems with a slow decay in the
Kolmogorov $n$-width, such as advection dominated problems that feature
sharp gradients, moving shock fronts and turbulence
\cite{Greif2019,Peherstorfer2022}.
One approach toward addressing this issue has been to construct local
reduced models that approximate the original dynamics for a specific
subset of the considered parameter-time domain
\cite{Amsallem2012,Peherstorfer2020aadeim}.
A reduction approach based on time windows was introduced in
\cite{Parish2021,Shimizu2021}
to build temporally local reduced models that are small but accurate
over a short time period for advection dominated problems.
This approach has since been further extended and applied to problems
in various domains
\cite{Copeland2022,Cheung2023,Cheung2024}.

The methods outlined above form a subspace of the original
state space and define the reduced state variables by linear
projection.
An alternative approach is to construct a nonlinear projection
of the original state variables onto an appropriate subspace,
often called a latent space.
The projection map is often represented by a neural
network, with the latent space defined by the architecture
of the employed neural network and its training procedure
\cite{Lee2020,Maulik2021,Fries2022,YKim2022}.

In addition to accuracy and reduced complexity, it is often
desirable for reduced models to preserve aspects of the
qualitative structure of the original dynamics,
such as conserved quantities or symplectic structure.
Several reduction schemes have been proposed that incorporate
structure preservation constraints directly into the
formulation of the reduced model, ensuring that the
relevant structure is not destroyed in the reduction process
\cite{Farhat2014,Kalashnikova2014,Farhat2015,Peng2016,Gong2017,
Carlberg2018,Greydanus2019,Swischuk2019,Hernandez2021,Hesthaven2022,
Sharma2023,Bajars2025,Gruber2025}.

\subsection{Hyper-reduction methods}
We consider a dynamical system with state space
$\Rbb^N$, $N\in\Nbb$,
governed by the differential equation
$\dot{u}=f(u)$
with state vector $u\in \Rbb^N$
and nonlinear function $f\colon\Rbb^N\to\Rbb^N$.
We assume access to an orthonormal reduced basis matrix
$\Phi\in\Rbb^{N\times n}$ with $n\ll N$,
where each column represents a reduced basis function.
Using the representation
$u\approx\Phi\hat{u}$, $\hat{u}\in\Rbb^n$,
we derive the governing equation for the reduced state
vector $\hat{u}$ by projection
\begin{equation*}
    \dot{\hat{u}}=\Phi^\top f(u)=\Phi^\top f(\Phi\hat{u}).
\end{equation*}
Although this is an equation for the $n\ll N$ reduced state
variables, the evaluation of the nonlinear function $f$ 
at every timestep depends on $\Phi\hat{u}\in\Rbb^N$
and is expected to dominate the model's computational cost.
For this reason, various \emph{hyper-reduction} methods
have been developed to approximate the evaluation of
$f$ in different ways.

Hyper-reduction methods can be divided into interpolation and
quadrature based methods.
Interpolation methods employ a set of collocation points
to approximate the evaluation of nonlinear function $f$,
an idea that can be traced back to \cite{Everson1995} 
in the context of image reconstruction. 
Examples of this class of methods include the
empirical interpolation method \cite{Barrault2004}
and its discrete variant \cite{Chaturantabut2010},
and employ various sampling strategies to select the collocation
points based on different optimality criteria 
\cite{Drmac2016,Drmac2018,Peherstorfer2020deim,Lauzon2024}.
Quadrature methods offer an alternative strategy for
efficiently approximating integrals of nonlinear
functions arising in finite element models.
They derive reduced quadrature rules that approximate
the spatial integration of functions over the computational
domain by requiring their evaluation at a limited
number of quadrature points.
The origin of this approach can be traced back to \cite{An2008}
in the context of computer graphics applications. 
Examples in this class of methods include
the energy conserving sampling and weighting method
\cite{Farhat2014,Farhat2015},
the empirical cubature method
\cite{Hernandez2017},
the EQP method
\cite{Patera2017,Yano2019,Du2022}
and more
\cite{Ryu2015,Devore2019}. 

\subsection{Lagrangian hydrodynamics}
We consider the compressible Euler equations of fluid dynamics
in a three dimensional spatial domain,
which can be used to model inviscid, high speed flows and
shock wave propagation
\cite{Serrin1959,Aris1989,Dobrev2012}.
The numerical solution of the equations can generally be achieved
by two classes of methods: Eulerian and Lagrangian methods.
In Eulerian methods the unknown variables are treated as
fields over a fixed computational mesh,
whereas in Lagrangian methods the mesh moves according to
the local fluid velocity
\cite{Benson1992}.

Lagrangian methods have traditionally been formulated using
approaches such as the staggered grid hydrodynamics
and the cell centered hydrodynamics methods,
which usually employ finite difference or finite volume
spatial discretizations.
The accuracy of these methods relies heavily on the quality
of the underlying mesh, which moves and deforms over time
and can lead to phenomena such as mesh tangling.
One way to address this issue has been the use of the
arbitrary Lagrangian-Eulerian method, which involves a
Lagrangian phase followed by a remesh and remap phase.
In the Lagrangian phase, the computational mesh evolves
according to the fluid velocity until its quality deteriorates
in a prescribed sense.
Once that happens, the mesh is adjusted accordingly and
the solution fields are remapped onto the adjusted mesh,
allowing the Lagrangian simulation to resume.
We refer the reader to \cite{Dobrev2012}
for details on the methods outlined above and their
historical development.

In this work we employ the Lagrangian method developed in
\cite{Dobrev2012},
which uses a high order finite element discretization of the
Euler equations in two and three dimensional cartesian grids.
Among other features, the employed method offers strong
conservation of mass and discrete total energy.
We refer the reader to \cite{Dobrev2012,Dobrev2013}
for the detailed development of the method and its
application to various benchmark problems.

\subsection{Main results}
We resume the line of research initiated in
\cite{Copeland2022,Cheung2023},
where the authors derived
projection based reduced models for Lagrangian hydrodynamics
problems using interpolation hyper-reduction methods.
In the present work we use quadrature hyper-reduction
methods, which are naturally well suited to the finite element
discretization used in the full model.
The main features of this work can be summarized as follows.
\begin{enumerate}
\item\emph{Projection based reduction}. 
Starting from a finite element (FEM) discretization of the
compressible Euler equations, we develop projection based
reduction models using reduced basis functions computed
from simulation data.
We derive two reduced problem formulations, demonstrating
that the reduction process corresponds to replacing the original
FEM trial and test functions with the reduced basis functions.
We use the EQP method to reduce the cost of the spatial integration
of nonlinear force functions that arises in the FEM discretization
process.
\item\emph{Conservative quadrature}.
We develop an energy conservative variant of the basic
EQP method that enforces exact conservation of the discrete
total energy in the strong sense in the hyper-reduced model.
This is achieved by deriving a reduced quadrature rule
that preserves the energy conservation property of the full model.
\item\emph{Applications}.
We implement the hyper-reduced models derived using the basic
and energy conservative EQP methods in the
\texttt{Laghos}\footnote{
Code repository: https://github.com/CEED/Laghos}
Lagrangian hydrodynamics simulation code using the
\texttt{libROM}\footnote{
Code repository: https://github.com/LLNL/libROM}
library for model reduction.
We apply the methods to four standard benchmark problems,
reporting their accuracy and obtained computational speedup.
Our numerical results show that the numerical implementation
of the energy conservative EQP method leads to conservation of total
energy to near machine precision for all simulated problem cases.
\end{enumerate}

\section{Full model}\label{sec:fom}
We consider the system of Euler equations of fluid dynamics in
a Lagrangian reference frame with no external body forces
\begin{equation}\label{eq:euler}
\begin{aligned}
    \text{mass conservation}:& &\frac{1}{\rho}\frac{d\rho}{dt}
        &=-\nabla\cdot v\\
    \text{momentum conservation}:& &\rho\frac{dv}{dt}
        &=\nabla\cdot\sigma\\
    \text{energy conservation}:& &\rho\frac{de}{dt}
        &=\sigma :\nabla v\\
    \text{equation of motion}:& &\frac{dx}{dt}&=v
\end{aligned}
\end{equation}
where $d/dt$ denotes the material derivative, 
$\rho$ the density of the fluid,
$x$ and $v$ the position and velocity of the
particles in a deformable medium $\Omega(t)$ in the
Eulerian coordinates,
$\sigma$ the deformation stress tensor, and
$e$ the internal energy per unit mass.
These physical quantities are treated as functions of time
$t\geq0$ and the initial configuration of particles
$x_0\in\Omega_0=\Omega(0)$
\cite{Dobrev2012,Aris1989,Serrin1959}.
The evolution equation for the position field $x(t,x_0)$
governs the Lagrangian advection of the fluid particles;
in the spatially discretized system, this leads to a
computational mesh that deforms according to the local
fluid velocity.

We use the isotropic stress tensor
$\sigma=-pI+\sigma_a$,
where $p$ denotes the thermodynamic pressure and
$\sigma_a$ the artificial viscosity stress.
The thermodynamic pressure is governed by a constitutive equation
and can be expressed as a function of the density and internal
energy.
In the present work we focus on the case of polytropic ideal
gases with an adiabatic index $\gamma>1$,
which yields the constitutive equation
\begin{equation}\label{eq:EOS}
  p=(\gamma-1)\rho e.
\end{equation}
The system is prescribed with an initial condition and the
wall boundary condition $v\cdot n=0$,
where $n$ denotes the outward normal unit vector on the
domain boundary.

\subsection{Spatial discretization}
We follow \cite{Dobrev2012}
to derive a semidiscrete Lagrangian variational formulation based on
a computational mesh that deforms according to the fluid velocity. 
By the Reynolds transport theorem, the mass conservation in
\eqref{eq:euler}
is related to the Jacobian $\nabla_{x_0} x(t,x_0)$
of the Lagrangian transformation $x(t,x_0)$ by 
\begin{equation*}
    \rho(t,x_0)=\frac{\rho(0,x_0)}{\vert J(t,x_0)\vert}
\end{equation*}
where $J(t,x_0)=\nabla_{x_0}x(t,x_0)$
is the deformation gradient and
$\vert J\vert$ its determinant
\cite{Dobrev2012}.
The above equation can be used to determine the density field for
all $t\geq 0$.
As a result, from now on we focus on the differential equations
for the state variable $w=(v,e,x)$.

For the spatial discretization of
\eqref{eq:euler}
we employ a finite element method (FEM)
using a kinematic space
$\Vcal\subset[H^1(\Omega_0)]^d$
with spatial dimension $d\in\Nbb$
and basis $\{\theta_{v,i}\}_{i=0}^{N_v-1}$
for approximating the position and velocity fields,
and a thermodynamic space
$\Ecal\subset L_2(\Omega_0)$ 
with basis $\{\theta_{e,i}\}_{i=0}^{N_e-1}$
for approximating the energy field,
where $N_v$, $N_e\in\Nbb$ denote the number of
the global degrees of freedom in the corresponding spaces.
In what follows, the FEM basis functions are used as both trial
and test functions.

We approximate the full solution fields by
\begin{equation*}
\begin{aligned}
    v(t,x_0)\approx v_f(t,x_0)&=
        \sum_{i=0}^{N_v-1}v_i(t)\theta_{v,i}(x_0)\\
    e(t,x_0)\approx e_f(t,x_0)&=
        \sum_{i=0}^{N_e-1}e_i(t)\theta_{e,i}(x_0)\\
    x(t,x_0)\approx x_f(t,x_0)&=
        \sum_{i=0}^{N_v-1}x_i(t)\theta_{v,i}(x_0).
\end{aligned}
\end{equation*}
We use
$\vect{v}$, $\vect{x}\in\Rbb^{N_v}$ and
$\vect{e}\in\Rbb^{N_e}$
to denote the FEM coefficient vectors of the corresponding
solution fields and
$\vect{w}=(\vect{v},\vect{e},\vect{x})^\top\in\Rbb^N$,
$N=2N_v+N_e$
to denote the full state vector.
The semidiscrete Lagrangian conservation laws can be derived by
testing equations \eqref{eq:euler} 
against the Lagrangian extensions of the FEM basis functions
on $\Omega(t)$
\begin{equation}\label{eq:fom-semi}
\begin{aligned}
    \text{momentum conservation}:& &\matr{M}_v\frac{d\vect{v}}{dt}&=
        -\vect{F}_v(\vect{w})\\
    \text{energy conservation}:& &\matr{M}_e\frac{d\vect{e}}{dt}&=
        \vect{F}_e(\vect{w},\vect{v})\\
    \text{equation of motion}:& &\frac{d\vect{x}}{dt}&=\vect{v}
\end{aligned}
\end{equation}
with mass matrices $\matr{M}_v\in\Rbb^{N_v\times N_v}$
and $\matr{M}_e\in\Rbb^{N_e\times N_e}$
\begin{equation}
    M_{v,ij}=\int_{\Omega(t)}\rho\theta_i^v\theta_j^vdx\qquad
    M_{e,ij}=\int_{\Omega(t)}\rho\theta_i^e\theta_j^edx
\end{equation}
and vector valued nonlinear force functions
$\vect{F}_v:\Rbb^N\to\Rbb^{N_v}$ and
$\vect{F}_e:\Rbb^N\times\Rbb^{N_v}\to\Rbb^{N_e}$
\begin{equation}
    F_{v,i}(\vect{w})=\int_{\Omega(t)}\sigma(w_f):\nabla\theta_i^vdx\qquad
    F_{e,i}(\vect{w},\vect{v}')=
        \int_{\Omega(t)}(\sigma(w_f):\nabla v')\,\theta_i^edx.
\end{equation}
Introducing the matrix valued force function
$\matr{F}:\Rbb^N\rightarrow\Rbb^{N_v\times N_e}$ 
with entries
\begin{equation}
    F_{ij}(\vect{w})=
        \int_{\Omega(t)}(\sigma(w_f):\nabla\theta_i^v)\,\theta_j^edx
\end{equation}
the nonlinear force functions can be compactly represented as 
\begin{equation}\label{eq:force-vec}
    \vect{F}_v(\vect{w})=\matr{F}(\vect{w})\,\vect{1}_\Ecal\qquad
    \vect{F}_e(\vect{w},\vect{v}')=\matr{F}(\vect{w})^\top\vect{v}'
\end{equation}
with $\vect{1}_\Ecal\in\mathbb{R}^{N_e}$ the coefficient vector of the
unity function in $\Ecal$.

In the above spatial integrations we write the integrals as
taking place over the control volume
$\Omega(t)$ instead of $\Omega_0=\Omega(0)$.
For any scalar valued field $a(t,x_0)$, $x_0\in\Omega_0$,
this is equivalent to using its Lagrangian extension
$a(t,x(t,x_0))$, $x(t,x_0)\in\Omega(t)$,
to perform the spatial integration; without misunderstanding,
we use the same symbol to denote the two field descriptions.
The two approaches are connected by performing the change of
variables $x(t,x_0)\mapsto x_0$
\begin{equation*}
    \int_{\Omega(t)}a(t,x(t,x_0))dx(t,x_0)=
        \int_{\Omega_0}a(t,x_0)\vert J(t,x_0)\vert dx_0
\end{equation*}
where $\vert J(t,x_0)\vert$
is the determinant of the deformation gradient defined
in the previous section.
To simplify the notation we use integrations over
$\Omega(t)$ throughout this work, implicitly understanding that
the above change of variables is carried out to perform
the integrations over $\Omega_0$.

\subsection{Temporal discretization}\label{sec:fom-time-discretize}
The time domain is discretized as
$\{t_k\}_{k=0}^{K-1}$
with $t_0=0$ and $t_{K-1}=t_f$.
All quantities at time $t_k$ are denoted by subscript $k$.
We use the two-stage average Runge-Kutta (RK2A) scheme derived
in \cite{Dobrev2012},
which has been shown to conserve the discrete total energy of the system.
Applying the RK2A scheme leads to the discrete system
of equations
\begin{equation}\label{eq:fom-rk2avg}
\begin{aligned}
    \vect{v}_{k+1/2}&=\vect{v}_k-(\Delta t_k/2)
        \matr{M}_v^{-1}\vect{F}_v(\vect{w}_k)&
    \vect{v}_{k+1}&=\vect{v}_k-\Delta t_k
        \matr{M}_v^{-1}\vect{F}_v(\vect{w}_{k+1/2})\\
    \vect{e}_{k+1/2}&=\vect{e}_k+(\Delta t_k/2)
        \matr{M}_e^{-1}\vect{F}_e(\vect{w}_k,\vect{v}_{k+1/2})&
    \vect{e}_{k+1}&=\vect{e}_k+\Delta t_k
        \matr{M}_e^{-1}\vect{F}_e(\vect{w}_k,\bar{\vect{v}}_{k+1/2})\\
    \vect{x}_{k+1/2}&=\vect{x}_k+
        (\Delta t_k/2)\vect{v}_{k+1/2}&
    \vect{x}_{k+1}&=\vect{x}_k+
        \Delta t_k\bar{\vect{v}}_{k+1/2}
\end{aligned}
\end{equation}
with
$\bar{\vect{v}}_{k+1/2}=(\vect{v}_k+\vect{v}_{k+1})/2$.
To ensure stability of the explicit timestepping scheme we also
employ the timestep control algorithm described in
\cite{Dobrev2012}.

\section{Reduced model}\label{sec:rom}
In this section we derive the projection based reduced
model for the semidiscrete Lagrangian conservation laws
\eqref{eq:fom-semi}. 
We introduce the reduced bases for velocity
$\{\phi_{v,i}\}_{i=0}^{n_v-1}\subset\Vcal$,
energy
$\{\phi_{e,i}\}_{i=0}^{n_e-1}\subset\Ecal$
and position
$\{\phi_{x,i}\}_{i=0}^{n_x-1}\subset\Vcal$
with reduced sizes
$n_v$, $n_x\ll N_v$ and $n_e\ll N_e$.
To approximate each solution field we restrict our solution space
to the linear span of each basis
\begin{equation}\label{eq:rom-sol}
\begin{aligned}
    v(t,x_0)\approx\tilde{v}(t,x_0)&=v_{os}(x_0)
        +\sum_{i=0}^{n_v-1}\hat{v}_i(t)\phi_{v,i}(x_0)\\
    e(t,x_0)\approx\tilde{e}(t,x_0)&=e_{os}(x_0)
        +\sum_{i=0}^{n_e-1}\hat{e}_i(t)\phi_{e,i}(x_0)\\
    x(t,x_0)\approx\tilde{x}(t,x_0)&=x_{os}(x_0)
        +\sum_{i=0}^{n_x-1}\hat{x}_i(t)\phi_{x,i}(x_0)
\end{aligned}
\end{equation}
with offset fields
$v_{os}$, $x_{os}\in\Vcal$ and $e_{os}\in\Ecal$.

The FEM coefficient vectors of the reduced basis functions are
represented by the orthonormal basis matrices
$\matr{\Phi}_v\in\Rbb^{N_v\times n_v}$,
$\matr{\Phi}_e\in\Rbb^{N_e\times n_e}$ and
$\matr{\Phi}_x\in\Rbb^{N_v\times n_x}$,
where each column holds the FEM coefficients of one basis function.
We compute the reduced basis matrices from simulation data of the
full model using the POD method \cite{Hinze2005}.
To decide how many POD basis functions to use to form each basis
matrix we employ the POD energy ratio
\begin{equation}\label{eq:POD-energy}
    \frac{\sum_{i=0}^{n_k-1}\sigma_i}{\sum_{i=0}^{N_{snap}-1}\sigma_i}
        \geq e_\sigma
\end{equation}
where $\sigma_i\geq 0$ denotes the $i$th largest singular value
of the employed data matrix consisting of $N_{snap}$
solution snapshots,
$e_\sigma\geq 0$ a chosen threshold
and $n_k$ the size of the corresponding reduced state
vector (velocity $v$, energy $e$ or position $x$).
Using the introduced reduced basis matrices we rewrite
\eqref{eq:rom-sol}
in terms of FEM coefficient vectors
\begin{equation}
\begin{aligned}
    \vect{v}(t)\approx\tilde{\vect{v}}(t)&=\vect{v}_{os}
        +\matr{\Phi}_v\hat{\vect{v}}(t)\\
    \vect{e}(t)\approx\tilde{\vect{e}}(t)&=\vect{e}_{os}
        +\matr{\Phi}_e\hat{\vect{e}}(t)\\
    \vect{x}(t)\approx\tilde{\vect{x}}(t)&=\vect{x}_{os}
        +\matr{\Phi}_x\hat{\vect{x}}(t) 
\end{aligned}
\end{equation}
where $\hat{\vect{v}}\in\Rbb^{n_v}$, 
$\hat{\vect{e}}\in\Rbb^{n_e}$ and
$\hat{\vect{x}}\in\Rbb^{n_x}$
denote the reduced solution state vectors.
We also use
$\tilde{\vect{w}}=(\tilde{\vect{v}},\tilde{\vect{e}},
\tilde{\vect{x}})^\top\in\Rbb^N$ 
to denote the FEM coefficient vector of
$\tilde{w}=(\tilde{v},\tilde{e},\tilde{x})$.

Making the above substitutions in \eqref{eq:fom-semi}
we reach the overdetermined semidiscrete system
\begin{equation}\label{eq:rom-semi}
\begin{aligned}
    \matr{M}_v\matr{\Phi}_v\frac{d\hat{\vect{v}}}{dt}&=
        -\vect{F}_v(\tilde{\vect{w}})\\
    \matr{M}_e\matr{\Phi}_e\frac{d\hat{\vect{e}}}{dt}&=
        \vect{F}_e(\tilde{\vect{w}},\tilde{\vect{v}})\\
    \matr{\Phi}_x\frac{d\hat{\vect{x}}}{dt}&=\tilde{\vect{v}}
\end{aligned}
\end{equation}
with the appropriate initial condition at $t=0$.
The above corresponds to using the reduced basis functions as
the trial functions and the FEM basis functions as the test
functions.
Next we present the derivation of two alternative reduced systems
of equations for the reduced state vector
$\hat{\vect{w}}=(\hat{\vect{v}},\hat{\vect{e}},\hat{\vect{x}})
\in\Rbb^n$,
$n=n_v+n_e+n_x$.
The reason for introducing two alternative systems is that both
systems will be employed by the hyper-reduction methods introduced
in the following sections.

Starting from \eqref{eq:rom-semi},
the first reduced system is derived by inverting the mass matrices
and transpose multiplying by the reduced basis matrices
\begin{equation}\label{eq:rom-beqp}
\begin{aligned}
    \frac{d\hat{\vect{v}}}{dt}&=
        -\matr{\Psi}_v^\top\vect{F}_v(\tilde{\vect{w}})=
        \vect{F}_v^\psi(\tilde{\vect{w}})\\
    \frac{d\hat{\vect{e}}}{dt}&=
        \matr{\Psi}_e^\top\vect{F}_e(\tilde{\vect{w}},\tilde{\vect{v}})=
        \vect{F}_e^\psi(\tilde{\vect{w}},\tilde{\vect{v}})\\
    \frac{d\hat{\vect{x}}}{dt}&=\matr{\Phi}_x^\top\tilde{\vect{v}}
\end{aligned}
\end{equation}
with
$\matr{\Psi}_v=\matr{M}_v^{-1}\matr{\Phi}_v\in\Rbb^{N_v\times n_v}$
and 
$\matr{\Psi}_e=\matr{M}_e^{-1}\matr{\Phi}_e\in\Rbb^{N_e\times n_e}$. 
This corresponds to using the functions represented by the columns of
$\matr{\Psi}_v$ and $\matr{\Psi}_e$
as the test functions.
More precisely, the components of force vectors
$\vect{F}_v^\psi\in\Rbb^{n_v}$ and
$\vect{F}_e^\psi\in\Rbb^{n_e}$
are given by the integrals
\begin{equation}\label{eq:rom-force-beqp}
    F_{v,i}^\psi(\tilde{\vect{w}})=
        \int_{\Omega(t)}\sigma(\tilde{w}):\nabla\psi^v_idx\qquad
    F_{e,i}^\psi(\tilde{\vect{w}},\tilde{\vect{v}}')=
        \int_{\Omega(t)}(\sigma(\tilde{w}):\nabla\tilde{v}')\,\psi^e_idx
\end{equation}
where functions $\psi^v_i\in\Vcal$ and $\psi^e_i\in\Ecal$
have the $i$th column of matrices
$\matr{\Psi}_v$ and $\matr{\Psi}_e$ respectively
as their FEM coefficient vectors.

The second reduced system is derived by directly transpose multiplying
by the reduced basis matrices in
\eqref{eq:rom-semi}
\begin{equation}\label{eq:rom-ceqp}
\begin{aligned}
    \hat{\matr{M}}_v\frac{d\hat{\vect{v}}}{dt}&=
        -\matr{\Phi}_v^{\top}\vect{F}_v(\tilde{\vect{w}})=
        \vect{F}_v^\phi(\tilde{\vect{w}})\\
    \hat{\matr{M}}_e\frac{d\hat{\vect{e}}}{dt}&=
        \matr{\Phi}_e^{\top}\vect{F}_e(\tilde{\vect{w}},\tilde{\vect{v}})=
        \vect{F}_e^\phi(\tilde{\vect{w}},\tilde{\vect{v}})\\
    \frac{d\hat{\vect{x}}}{dt}&=\matr{\Phi}_x^{\top}\tilde{\vect{v}}
\end{aligned}
\end{equation}
with reduced mass matrices
\begin{equation}\label{eq:rmass-mat}
    \hat{\matr{M}}_v=\matr{\Phi}_v^{\top}\matr{M}_v
        \matr{\Phi}_v\in\Rbb^{n_v\times n_v}\qquad
    \hat{\matr{M}}_e=\matr{\Phi}_e^{\top}\matr{M}_e
        \matr{\Phi}_e\in\Rbb^{n_e\times n_e}.
\end{equation}
The above corresponds to using the reduced basis functions as both trial
and test functions.
More specifically, the reduced mass matrices are given by the integrals
\begin{equation}
    \hat{M}_{v,ij}=\int_{\Omega(t)}\rho\phi^v_i\phi^v_jdx\qquad
    \hat{M}_{e,ij}=\int_{\Omega(t)}\rho\phi^e_i\phi^e_jdx
\end{equation}
and the force vectors
$\vect{F}_v^\phi\in\Rbb^{n_v}$ and
$\vect{F}_e^\phi\in\Rbb^{n_e}$
by
\begin{equation}\label{eq:rom-force-ceqp}
    F_{v,i}^\phi(\tilde{\vect{w}})=
        \int_{\Omega(t)}\sigma(\tilde{w}):\nabla\phi^v_idx\qquad
    F_{e,i}^\phi(\tilde{\vect{w}},\tilde{\vect{v}}')=
        \int_{\Omega(t)}(\sigma(\tilde{w}):\nabla\tilde{v}')\,\phi^e_idx.
\end{equation}

Although the size of the reduced systems
\eqref{eq:rom-beqp} and \eqref{eq:rom-ceqp}
has been successfully reduced to $n\ll N$,
the nonlinearity of the force functions in the right hand
side of the equations means that they still depend on the
full state vector
$\tilde{\vect{w}}\in\Rbb^N$.
As a result, evaluating the force functions requires
lifting the reduced state
$\hat{\vect{w}}\in\Rbb^n$ to $\tilde{\vect{w}}$
at every timestep of the integration.
The cost of this lifting operation dominates the computational
cost of the simulation and prevents its acceleration.
To address this issue, a hyper-reduction method
is employed to derive a hyper-reduced model that can be
simulated without requiring lifting to the full state vector
$\tilde{\vect{w}}$.

In this work we employ the EQP hyper-reduction method,
which uses reduced quadrature rules to evaluate the nonlinear
force functions with reduced cost
\cite{Patera2017,Yano2019}.
Quadrature methods such as EQP are naturally well suited to
FEM spatial discretization procedures, such as the one employed
in this work.
In that respect, the present work can be viewed as a continuation
of \cite{Copeland2022,Cheung2023},
where interpolation methods were used to derive
hyper-reduced Lagrangian hydrodynamics models.

\subsection{Empirical quadrature procedure}\label{sec:eqp}
In this section we describe the EQP hyper-reduction method
using the model equations \eqref{eq:rom-beqp}.
The computation of the force vectors
$\vect{F}_v^\psi$ and $\vect{F}_e^\psi$
of \eqref{eq:rom-beqp}
requires the spatial integrations
\begin{equation}
    \vect{F}_v^\psi(\tilde{\vect{w}})=
        \int_{\Omega(t)}\vect{g}_v^\psi(x,\tilde{w})dx\qquad
    \vect{F}_e^\psi(\tilde{\vect{w}},\tilde{\vect{v}})=
        \int_{\Omega(t)}\vect{g}_e^\psi(x,\tilde{w},\tilde{v})dx
\end{equation}
with the integrands
$\vect{g}_v^\psi$ and $\vect{g}_e^\psi$
taking values in $\Rbb^{n_v}$ and $\Rbb^{n_e}$
respectively, with components
\begin{equation}\label{eq:rom-beqp-force-integrands}
    g_{v,i}^\psi(x,\tilde{w})=\sigma(\tilde{w})(x):\nabla\psi^v_i(x)\qquad
    g_{e,i}^\psi(x,\tilde{w},\tilde{v})=
        (\sigma(\tilde{w})(x):\nabla \tilde{v}(x))\,\psi^e_i(x). 
\end{equation}
These integrations can be approximated using numerical quadrature rules
employing the underlying FEM basis functions
\cite{Dobrev2012,Brenner2008}.
Let
$\{x^v_j,\rho^v_j\}_{j=0}^{J_v-1}$ and
$\{x^e_j,\rho^e_j\}_{j=0}^{J_e-1}$
denote the full quadrature rules 
($x_j$ denoting points, $\rho_j\geq 0$ weights) for velocity
and energy, with $J_v$ and $J_e$ their respective sizes.
We can approximate the integrations by
\begin{equation}\label{eq:full-quad-beqp}
    \vect{F}_v^\psi(\tilde{\vect{w}})\approx
        \sum_{j=0}^{J_v-1}\rho^v_j\vect{g}_v^\psi(x^v_j,\tilde{w})\qquad
    \vect{F}_e^\psi(\tilde{\vect{w}},\tilde{\vect{v}})\approx
        \sum_{j=0}^{J_e-1}\rho^e_j\vect{g}_e^\psi(x^e_j,\tilde{w},\tilde{v}).
\end{equation}
The computational cost of the above summations is proportional to
the sizes of the full quadrature rules ($J_v$ and $J_e$),
since the integrands must be evaluated at every point of
the corresponding quadrature rule.
In addition, since the integrands depend on the state variables
they must be updated at every stage of the time integration.

To remove the dependence of the computational cost on the full sizes
$J_v$ and $J_e$ we employ the EQP hyper-reduction method
\cite{Patera2017,Yano2019}.
EQP is used to generate quadrature rules of reduced size that can
approximate the evaluation of the force functions with reduced
computational cost while maintaining a reasonable level
of accuracy.
More specifically, EQP seeks to construct new quadrature rules 
$\{x^v_j,\tilde{\rho}^v_j\}_{j=0}^{J_v-1}$ and
$\{x^e_j,\tilde{\rho}^e_j\}_{j=0}^{J_e-1}$
where only $\hat{J}_v\ll J_v$ and $\hat{J}_e\ll J_e$
of the new weights $\tilde{\rho}^v_j$ and $\tilde{\rho}^e_j$
are nonzero.
In this way, computing the sums in \eqref{eq:full-quad-beqp}
requires the evaluation of the integrands at only
$\hat{J}_v$ and $\hat{J}_e$
sampled quadrature points respectively, which are the points
corresponding to nonzero weights.

Note that EQP does not attempt to identify new quadrature points;
rather, it samples a subset of the original quadrature points with
appropriately modified weights.
In that sense EQP can be thought of as a quadrature analog of
interpolation hyper-reduction methods based on gappy POD
\cite{Larsson2026}.
The problem of constructing the reduced quadrature rules is posed as a
linear optimization problem
\cite{Patera2017}.
To outline the employed procedure we use the construction of
a reduced quadrature rule for the approximation of
$\vect{F}_v^\psi$
as an example.
The construction of a reduced rule for
$\vect{F}_e^\psi$
can be performed analogously.

Starting with the full quadrature rule
$\{x^v_j,\rho^v_j\}_{j=0}^{J_v-1}$
we want to find weights
$\{\tilde{\rho}^v_j\}_{j=0}^{J_v-1}$
that minimize the 1-norm
$\sum_{j=0}^{J_v-1}\tilde{\rho}^v_j$
subject to the nonnegativity constraints
$\tilde{\rho}^v_j\geq 0$, $j\in\{0,\ldots,J_v-1\}$,
and a total of $N_c^v\ll J_v$ accuracy constraints
\begin{equation*}
    \Bigl |\sum_{j=0}^{J_v-1}\rho^v_jg_{v,i}^\psi(x^v_j,\tilde{w}(t_k))
        -\sum_{j=0}^{J_v-1}\tilde{\rho}^v_j
        g_{v,i}^\psi(x^v_j,\tilde{w}(t_k))\Bigr |\leq\epsilon_s
\end{equation*}
where $0\leq k<N_t$ indexes the time snapshots
$\tilde{w}(t_k)$ used to form the constraints,
$0\leq i<n_v$ indexes the components of vector
$\vect{g}_v^\psi$,
and $0\leq s<N_c^v$ is the overall index of the accuracy
constraints with associated error thresholds $\epsilon_s>0$.
This implies a total of $N_c^v=n_vN_t$ constraints,
corresponding to $n_v$ constraints
(one for each component of $\vect{g}_v^\psi$)
per included snapshot of the state variables $\tilde{w}(t_k)$.
In addition to selecting the snapshots $\tilde{w}(t_k)$
that will be used to form the accuracy constraints,
selecting appropriate values for the error thresholds
$\epsilon_s$ is also important in constructing an effective
reduced quadrature rule.
In practice one seeks to strike a balance between
\emph{sparsity} (low number of nonzero weights, $\hat{J}_v$)
and \emph{accuracy} (low error thresholds $\epsilon_s$).

The linear optimization problem of forming the reduced quadrature
rule can be recast as an equivalent nonnegative least squares (NNLS)
problem \cite{Du2022,Sleeman2022}.
Details on the steps we take to form the NNLS problem
and solve it are given in \ref{app:implementation}.

\subsection{Basic EQP hyper-reduction}\label{sec:basic-eqp}
In this section we employ the EQP method as outlined above
to derive a hyper-reduced version of model
\eqref{eq:rom-beqp}.
In particular, we use the reduced quadrature rules
$\{x^v_j,\tilde{\rho}^v_j\}_{j=0}^{J_v-1}$ and
$\{x^e_j,\tilde{\rho}^e_j\}_{j=0}^{J_e-1}$
to define the respective reduced force functions
$\hat{\vect{F}}_v^\psi:\Rbb^n\to\Rbb^{n_v}$ and
$\hat{\vect{F}}_e^\psi:\Rbb^n\times\Rbb^{n_v}\to\Rbb^{n_e}$
\begin{equation}\label{eq:beqp-rforce}
    \hat{\vect{F}}_v^\psi(\hat{\vect{w}})=
        \sum_{j=0}^{J_v-1}\tilde{\rho}^v_j
        \vect{g}_v^\psi(x^v_j,\tilde{w})\qquad
    \hat{\vect{F}}_e^\psi(\hat{\vect{w}},\hat{\vect{v}})=
        \sum_{j=0}^{J_e-1}\tilde{\rho}^e_j
        \vect{g}_e^\psi(x^e_j,\tilde{w},\tilde{v})
\end{equation}
and semidiscrete hyper-reduced model
\begin{equation}\label{eq:hrom-beqp-semi}
\begin{aligned}
    \frac{d\hat{\vect{v}}}{dt}&=
        -\hat{\vect{F}}_v^\psi(\hat{\vect{w}})\\
    \frac{d\hat{\vect{e}}}{dt}&=
        \hat{\vect{F}}_e^\psi(\hat{\vect{w}},\hat{\vect{v}})\\
    \frac{d\hat{\vect{x}}}{dt}&=
        \matr{\Phi}_x^\top(\vect{v}_{os}+\matr{\Phi}_v\hat{\vect{v}}).
\end{aligned}
\end{equation}
The evaluation of the reduced force functions requires only
$\hat{J}_v\ll J_v$ and $\hat{J}_e\ll J_e$
evaluations of the respective integrands.
As a result, it requires knowledge only of the degrees of freedom
of $\tilde{w}$ sampled by the reduced quadrature rules.
To make this reduction in computational cost explicit in our notation,
the reduced force functions 
$\hat{\vect{F}}_v^\psi$ and $\hat{\vect{F}}_e^\psi$
are written as functions of the reduced state $\hat{\vect{w}}$,
instead of the lifted state $\tilde{\vect{w}}$.
For the evolution of $\hat{\vect{x}}$ in
\eqref{eq:hrom-beqp-semi},
the quantities $\matr{\Phi}_x^\top\vect{v}_{os}\in\Rbb^{n_x}$
and $\matr{\Phi}_x^\top\matr{\Phi}_v\in\Rbb^{n_x\times n_v}$
can be precomputed, leading to an evolution equation that depends
on reduced sizes only.

\subsection{Temporal discretization}\label{sec:rom-time-discretize}
Applying the RK2A scheme to the hyper-reduced system
\eqref{eq:hrom-beqp-semi},
we arrive at the discrete hyper-reduced system
\begin{equation}\label{eq:hrom-beqp-rk2avg}
\begin{aligned}
    \hat{\vect{v}}_{k+1/2}&=\hat{\vect{v}}_k-(\Delta t_k/2)
        \hat{\vect{F}}_v^\psi(\hat{\vect{w}}_k)&
    \hat{\vect{v}}_{k+1}&=\hat{\vect{v}}_k-\Delta t_k
        \hat{\vect{F}}_v^\psi(\hat{\vect{w}}_{k+1/2})\\
    \hat{\vect{e}}_{k+1/2}&=\hat{\vect{e}}_k+(\Delta t_k/2)
        \hat{\vect{F}}_e^\psi(\hat{\vect{w}}_k,\hat{\vect{v}}_{k+1/2})&
    \hat{\vect{e}}_{k+1}&=\hat{\vect{e}}_k+\Delta t_k
        \hat{\vect{F}}_e^\psi(\hat{\vect{w}}_{k+1/2},
        \bar{\hat{\vect{v}}}_{k+1/2})\\
    \hat{\vect{x}}_{k+1/2}&=\hat{\vect{x}}_k+(\Delta t_k/2)
        \matr{\Phi}_x^\top (\vect{v}_{os}
        +\matr{\Phi}_v\hat{\vect{v}}_{k+1/2})& 
    \hat{\vect{x}}_{k+1}&=\hat{\vect{x}}_k+\Delta t_k
        \matr{\Phi}_x^\top (\vect{v}_{os}
        +\matr{\Phi}_v\bar{\hat{\vect{v}}}_{k+1/2})
\end{aligned}
\end{equation}
with
$\bar{\hat{\vect{v}}}_{k+1/2}=(\hat{\vect{v}}_k
+\hat{\vect{v}}_{k+1})/2$.
In the above, the timestep control algorithm is enforced based
on the lifted state vector $\tilde{\vect{w}}$,
using again only the degrees of freedom sampled by the reduced
quadrature rules.
In general, this leads to timesteps that are different from those
used in the full order model, where the timestep control algorithm is
enforced using the full state vector $\vect{w}$. 

\subsection{Time windowing}
For unsteady problems dominated by advection, like the Euler equations
considered in the present work, the dimension
of the linear subspaces required to capture the relevant flow features
generally grows with the simulation time.
This is both because new dynamical behavior can emerge as the final
simulation time is increased, and because dynamically similar flow
features can appear in different regions of the spatial domain.
As a result, attempting to build a single reduced model for the whole
simulation time horizon may require a large reduced basis dimension to
achieve the desired degree of accuracy.
In turn, this can potentially eliminate any acceleration of the
simulation offered by the employed reduction method.

To address this issue we use the method of time windowing
\cite{Parish2021,Shimizu2021,Copeland2022,Cheung2023}.
With time windowing, we divide the considered simulation time interval
$[t_0,t_f]$ into $N_w$ subintervals (windows)
$[t_i, t_{i+1}]$, $0\leq i<N_w$,
and build a reduced model for each time window.
More specifically, for each window we collect the solution snapshots
that correspond to the prescribed time interval and use them to build
reduced basis matrices and reduced quadrature rules that are local to
the chosen time interval.
This allows us to build reduced bases and quadrature rules that are of
modest size and offer high accuracy in approximating the full solution
for each time subinterval.

\subsection{Simulation stages}
The overall simulation process consists of two main stages:
the offline and online stages.
In the offline stage we perform numerical simulations using the
full model \eqref{eq:fom-rk2avg}
and collect the desired solution snapshots.
Using those snapshots, we derive the reduced basis matrices
and reduced quadrature rules.
When multiple time windows are used, this process is carried out
for each time window independently as described earlier.
These steps complete the construction of the hyper-reduced model.

In the online stage we use the derived hyper-reduced model
to perform numerical simulations that depend on reduced sizes only,
thereby offering a reduction in computational cost.
In the present work we focus on reproductive simulations,
where the derived reduced model is used to reproduce
simulation data that has been used for its training.
This is in contrast to predictive simulations, where the reduced
model is tested in predicting previously unseen simulation data.

\section{Energy conservative EQP}\label{sec:ceqp}
In this section we develop a modified version of
the basic EQP method that enforces the additional constraint of
total energy conservation in the discrete hyper-reduced model.
To do so, we use the semidiscrete system \eqref{eq:rom-ceqp}.
Our goal is to derive reduced force functions
$\hat{\vect{F}}_v^\phi:\Rbb^n\to\Rbb^{n_v}$ and
$\hat{\vect{F}}_e^\phi:\Rbb^n\times\Rbb^{n_v}\to\Rbb^{n_e}$
to approximate the force function evaluations
\begin{equation*}
    \hat{\vect{F}}_v^\phi(\hat{\vect{w}})\approx
        \matr{\Phi}_v^\top\vect{F}_v(\tilde{\vect{w}})\qquad
    \hat{\vect{F}}_e^\phi(\hat{\vect{w}},\hat{\vect{v}})\approx
        \matr{\Phi}_e^\top\vect{F}_e(\tilde{\vect{w}},\tilde{\vect{v}})
\end{equation*}
with reduced computational cost.
Using the reduced force functions, we can then derive the
semidiscrete hyper-reduced model
\begin{equation}\label{eq:hrom-ceqp-semi}
\begin{aligned}
    \hat{\matr{M}}_v\frac{d\hat{\vect{v}}}{dt}&=
        -\hat{\vect{F}}_v^\phi(\hat{\vect{w}})\\
    \hat{\matr{M}}_e\frac{d\hat{\vect{e}}}{dt}&=
        \hat{\vect{F}}_e^\phi(\hat{\vect{w}},\hat{\vect{v}})\\
    \frac{d\hat{\vect{x}}}{dt}&=
        \matr{\Phi}_x^\top(\vect{v}_{os}+\matr{\Phi}_v\hat{\vect{v}})
\end{aligned}
\end{equation}
and apply the RK2A scheme to reach the discrete hyper-reduced system
\begin{equation}\label{eq:hrom-ceqp-rk2avg}
\begin{aligned}
    \hat{\vect{v}}_{k+1/2}&=\hat{\vect{v}}_k-(\Delta t_k/2)
        \hat{\matr{M}}_v^{-1}\hat{\vect{F}}_v^\phi(\hat{\vect{w}}_k)&
    \hat{\vect{v}}_{k+1}&=\hat{\vect{v}}_k-\Delta t_k
        \hat{\matr{M}}_v^{-1}\hat{\vect{F}}_v^\phi(\hat{\vect{w}}_{k+1/2})\\
    \hat{\vect{e}}_{k+1/2}&=\hat{\vect{e}}_k+(\Delta t_k/2)
        \hat{\matr{M}}_e^{-1}\hat{\vect{F}}_e^\phi
        (\hat{\vect{w}}_k,\hat{\vect{v}}_{k+1/2})&
    \hat{\vect{e}}_{k+1}&=\hat{\vect{e}}_k+\Delta t_k \hat{\matr{M}}_e^{-1}
        \hat{\vect{F}}_e^\phi(\hat{\vect{w}}_{k+1/2},
        \bar{\hat{\vect{v}}}_{k+1/2})\\
    \hat{\vect{x}}_{k+1/2}&=\hat{\vect{x}}_k+(\Delta t_k/2)
        \matr{\Phi}_x^\top (\vect{v}_{os}
        +\matr{\Phi}_v\hat{\vect{v}}_{k+1/2})&
    \hat{\vect{x}}_{k+1}&=\hat{\vect{x}}_k+\Delta t_k
        \matr{\Phi}_x^\top \vect{v}_{os}
        +\matr{\Phi}_v\bar{\hat{\vect{v}}}_{k+1/2}.
\end{aligned}
\end{equation}

In the following sections we define the reduced force functions
$\hat{\vect{F}}_v^\phi$ and $\hat{\vect{F}}_e^\phi$. 
We begin by identifying sufficient conditions for total energy
conservation in the strong sense.
Then we describe the implementation of the resulting hyper-reduction
scheme used to incorporate these conditions.

\subsection{Energy conservation}
The total energy of the system is the sum of its internal and
kinetic energy
\begin{equation*}
    TE(w)=IE(w)+KE(w)=
        \int_{\Omega(t)}\rho edx+
        \dfrac{1}{2}\int_{\Omega(t)}\rho\vert v\vert^2dx   
\end{equation*}
which can be represented using the FEM coefficient vectors
\begin{equation*}
    TE(\vect{w})=IE(\vect{w})+KE(\vect{w})=
        \vect{1}_\Ecal^\top\matr{M}_e\vect{e}+
        \dfrac{1}{2}\vect{v}^\top\matr{M}_v\vect{v}. 
\end{equation*}
The semidiscrete full model \eqref{eq:fom-semi}
conserves the total energy by design:
$\frac{d}{dt}TE(\vect{w}(t))=0$
for all $t\geq 0$.
In addition, the RK2A discrete scheme \eqref{eq:fom-rk2avg}
conserves the discrete total energy:
$TE(\vect{w}_{k+1})=TE(\vect{w}_k)$
for all $k\geq 0$ \cite{Dobrev2012}.
In the following theorem we derive sufficient conditions that
ensure that the conservation of discrete total energy is preserved
in the discrete hyper-reduced model
\eqref{eq:hrom-ceqp-rk2avg}.

\begin{theorem}\label{thm:energy-conserve}
Assume that the following three conditions are satisfied.
\begin{enumerate}
\item $\vect{v}_{os}=\vect{0}_\Vcal$
\item $\vect{1}_\Ecal=\matr{\Phi}_e\hat{\vect{1}}_e$
for some $\hat{\vect{1}}_e\in\Rbb^{n_e}$
\item $\hat{\vect{v}}^\top\hat{\vect{F}}_v^\phi(\hat{\vect{w}})=
\hat{\vect{1}}_e^\top\hat{\vect{F}}_e^\phi(\hat{\vect{w}},\hat{\vect{v}})$
for all $\hat{\vect{w}}\in\Rbb^{n}$.
\end{enumerate}
Then the discrete hyper-reduced model \eqref{eq:hrom-ceqp-rk2avg} 
conserves the discrete total energy; namely, 
\begin{equation*}
IE(\tilde{\vect{w}}_{k+1})+KE(\tilde{\vect{w}}_{k+1})=
    IE(\tilde{\vect{w}}_k)+KE(\tilde{\vect{w}}_k)
\end{equation*}
for all $k\geq 0$.
\end{theorem}
\begin{proof}
It holds true that
\begin{align*}
    IE(\tilde{\vect{w}}_{k+1})-IE(\tilde{\vect{w}}_k)&=
        \vect{1}_\Ecal^\top\matr{M}_e
        (\tilde{\vect{e}}_{k+1}-\tilde{\vect{e}}_k)\\
    &=\hat{\vect{1}}_e^\top\hat{\matr{M}}_e
        (\hat{\vect{e}}_{k+1}-\hat{\vect{e}}_k)\\
    &=\Delta t_k\hat{\vect{1}}_e^\top\hat{\vect{F}}_e^\phi
        (\hat{\vect{w}}_{k+1/2},\bar{\hat{\vect{v}}}_{k+1/2})\\
    &=\Delta t_k\bar{\hat{\vect{v}}}_{k+1/2}^\top
        \hat{\vect{F}}_v^\phi(\hat{\vect{w}}_{k+1/2})\\
    &=-\bar{\hat{\vect{v}}}_{k+1/2}^\top\hat{\matr{M}}_v
        (\hat{\vect{v}}_{k+1}-\hat{\vect{v}}_k)\\
    &=-\frac{1}{2}(\hat{\vect{v}}_{k+1}^\top\hat{\matr{M}}_v
        \hat{\vect{v}}_{k+1}-\hat{\vect{v}}_k^\top\hat{\matr{M}}_v
        \hat{\vect{v}}_k)\\
    &=-\frac{1}{2}(\tilde{\vect{v}}_{k+1}^\top\matr{M}_v
        \tilde{\vect{v}}_{k+1}-\tilde{\vect{v}}_k^\top
        \matr{M}_v\tilde{\vect{v}}_k)\\
    &=-\bigl(KE(\tilde{\vect{w}}_{k+1})-KE(\tilde{\vect{w}}_k)\bigr).
\end{align*}
\end{proof}

The three conditions presented in Theorem \ref{thm:energy-conserve}
are motivated by the physics and structure of the hyper-reduced model.
The first two conditions ensure that the kinetic energy $KE$
and internal energy $IE$ respectively can be represented in terms
of the reduced state vectors.
The third condition is the main structure preserving property coupling
the velocity and energy forcing terms, motivated by the fact that the
fluid flow governed by the discrete hyper-reduced model
\eqref{eq:hrom-ceqp-rk2avg} 
is driven by pressure-volume work.

\subsection{Implementation}\label{sec:ceqp-implementation}
To implement the energy conservative EQP method we begin with
the semidiscrete model \eqref{eq:rom-ceqp}.
We use zero offset vectors for position, velocity and energy
in \eqref{eq:rom-sol}, which ensures we satisfy condition 1
of Theorem \ref{thm:energy-conserve}.
To meet condition 2 we enrich the energy basis matrix
$\matr{\Phi}_e$ by adding the energy unit
$\vect{1}_\Ecal$ as a column vector.
Moreover, we orthonormalize the basis matrices
$\matr{\Phi}_v$ and $\matr{\Phi}_e$
such that the reduced mass matrices
\eqref{eq:rmass-mat}
reduce to identity matrices.

To satisfy condition 3 we are going to derive one combined reduced
quadrature rule to be used for both the velocity $\vect{F}_v^\phi$
and energy $\vect{F}_e^\phi$ force functions.
More specifically, we are going to form a reduced quadrature rule
for the force function
$\vect{F}^\phi:\Rbb^N\to\Rbb^{n_v\times n_e}$
\begin{equation}\label{eq:ec-rquad}
    \vect{F}^\phi(\tilde{\vect{w}})=\matr{\Phi}_v^\top
        \vect{F}(\tilde{\vect{w}})\matr{\Phi}_e=
        \int_{\Omega(t)}\matr{G}(x,\tilde{w})dx
\end{equation}
where $\matr{G}$ denotes the integrand taking values in
$\Rbb^{n_v\times n_e}$, with entries
\begin{equation}\label{eq:ceqp-force-integrand}
    G_{ij}(x,\tilde{w})=
        (\sigma(\tilde{w})(x):\nabla\phi^v_i(x))\,\phi^e_j(x).
\end{equation}
With respect to $\vect{F}^\phi$, the velocity and energy force
functions are given by
\begin{equation*}
    \vect{F}_v^\phi(\tilde{\vect{w}})=
        \vect{F}^\phi(\tilde{\vect{w}})\hat{\vect{1}}_e\qquad
    \vect{F}_e^\phi(\tilde{\vect{w}},\tilde{\vect{v}})
        =\vect{F}^\phi(\tilde{\vect{w}})^\top\hat{\vect{v}}.
\end{equation*}

We denote by
$\{x_j,\rho_j\}_{j=0}^{J-1}$
the full quadrature rule that can be used to approximate
the integration \eqref{eq:ec-rquad}
\begin{equation*}
    \vect{F}^\phi(\tilde{\vect{w}})\approx
        \sum_{j=0}^{J-1}\rho_j\matr{G}(x_j,\tilde{w}).
\end{equation*}
Starting with this rule, we employ the EQP method
to derive a reduced quadrature rule
$\{x_j,\tilde{\rho}_j\}_{j=0}^{J-1}$
with $\hat{J}\ll J$ nonzero weights, and use it to define
the reduced force function
$\hat{\vect{F}}^\phi:\Rbb^n\to\Rbb^{n_v\times n_e}$
\begin{equation}\label{eq:ceqp-quad}
    \hat{\vect{F}}^\phi(\hat{\vect{w}})=
        \sum_{j=0}^{J-1}\tilde{\rho}_j\matr{G}(x_j,\tilde{w}).
\end{equation}
The evaluation of
$\vect{F}_v^\phi$ and $\vect{F}_e^\phi$
is then approximated with reduced computational cost by
\begin{equation}\label{eq:hrforce-vec2}
    \hat{\vect{F}}_v^\phi(\hat{\vect{w}})=
        \hat{\vect{F}}^\phi(\hat{\vect{w}})\hat{\vect{1}}_e
        \approx\vect{F}_v^\phi(\tilde{\vect{w}})\qquad
    \hat{\vect{F}}_e^\phi(\hat{\vect{w}},\hat{\vect{v}})=
        \hat{\vect{F}}^\phi(\hat{\vect{w}})^\top\hat{\vect{v}}
        \approx\vect{F}_e^\phi(\tilde{\vect{w}},\tilde{\vect{v}}).
\end{equation}
It follows that for all $\hat{\vect{w}}\in\Rbb^n$
\begin{equation*}
    \hat{\vect{v}}^\top\hat{\vect{F}}_v^\phi(\hat{\vect{w}})=
        \hat{\vect{v}}^\top\hat{\vect{F}}^\phi(\hat{\vect{w}})
        \hat{\vect{1}}_e=
        \hat{\vect{1}}_e^\top\hat{\vect{F}}^\phi(\hat{\vect{w}})^\top
        \hat{\vect{v}}=
        \hat{\vect{1}}_e^\top\hat{\vect{F}}_e^\phi(\hat{\vect{w}},
        \hat{\vect{v}})
\end{equation*}
satisfying condition 3 of Theorem \ref{thm:energy-conserve}. 
Therefore, we conclude that the discrete hyper-reduced model
\eqref{eq:hrom-ceqp-rk2avg}
conserves the discrete total energy for all
time indices $k\geq 0$.
Additional details on the implementation of the energy
conservative EQP method are given in \ref{app:implementation}.

\section{Numerical experiments}\label{sec:experiments}
In this section we present numerical results for four problem cases
used to compare the performance of the basic and energy
conservative EQP methods.
Using each of the two EQP methods we simulate the following problems:
the 3D Sedov blast problem,
the 2D Gresho vortex problem,
the 3D triple point problem, and
the 3D Taylor-Green vortex problem.
They are all problems that have been used in the past as benchmarks
for testing the performance of different discretization and reduction
methods \cite{Dobrev2012,Copeland2022,Larsson2026}.

\subsection{Problems}
In the following paragraphs we provide an overview of each problem
case and its simulation parameters.
The spatial domain in each problem is discretized using an
initially uniform, cartesian mesh of hexahedral elements.
Table \ref{tab:discretize-params}
presents discretization parameters used for each of the
simulated problems, including the number of basis functions used for
each finite element space and the corresponding polynomial
basis orders.

\begin{table}[t]
    \centering
	\begin{tabular}{l l l l l}
	Run & R1 & R2 & R3 & R4\\
	\hline
	Problem & Sedov & Gresho & Triple point & Taylor-Green\\
	$m$ & 2 & 4 & 2 & 2\\
    $k$ & 2 & 2 & 2 & 2\\
    $N_v$ & 14739 & 18818 & 38475 & 14739\\
    $N_e$ & 4096 & 9216 & 10752 & 4096\\
    \hline
	\end{tabular}
	\caption{Spatial discretization parameters
    for each simulated problem.
    Parameter $m$ controls the mesh size $h=2^{-m}h_0$, where $h_0$
    denotes the coarsest mesh size;
    $k$ is the polynomial order of the basis used for the kinematic
    finite element space $\Vcal$, with order $k-1$ for the thermodynamic
    space $\Ecal$;
    $N_v$ and $N_e$ denote the global number of basis elements
    for the finite element spaces $\Vcal$ and $\Ecal$
    respectively.}
	\label{tab:discretize-params}
\end{table}

\begin{figure}[t]
     \centering
     \begin{subfigure}[b]{0.24\textwidth}
         \centering
         \includegraphics[width=\textwidth]{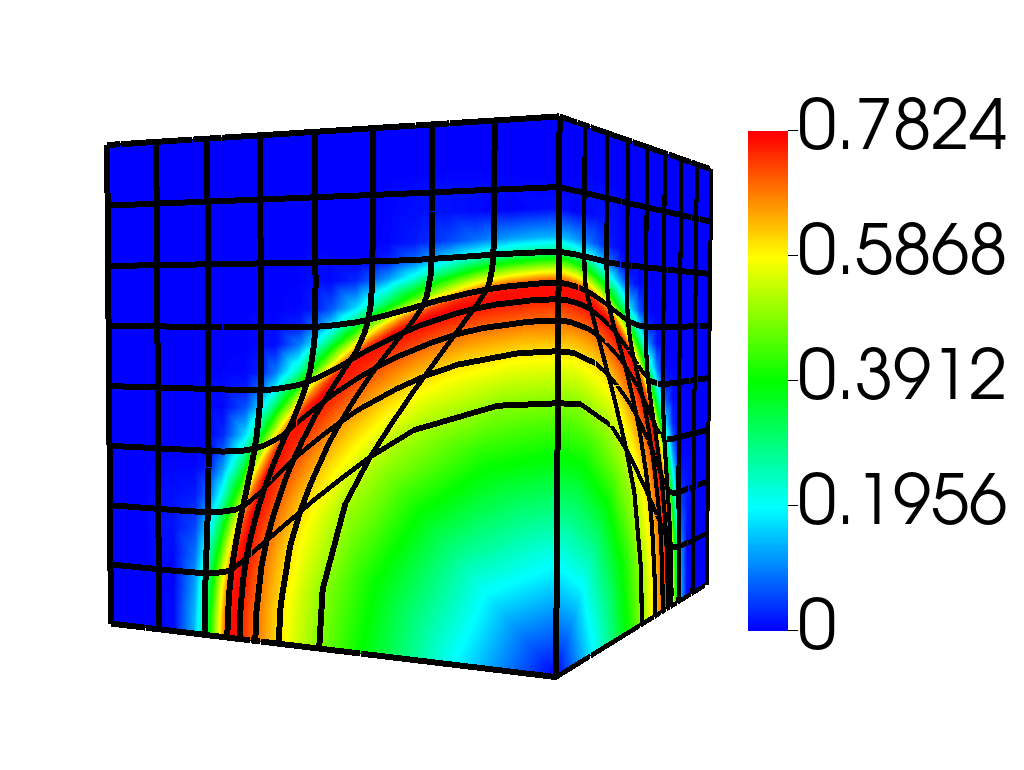}
     \end{subfigure}
     \hfill
     \begin{subfigure}[b]{0.24\textwidth}
         \centering
         \includegraphics[width=\textwidth]{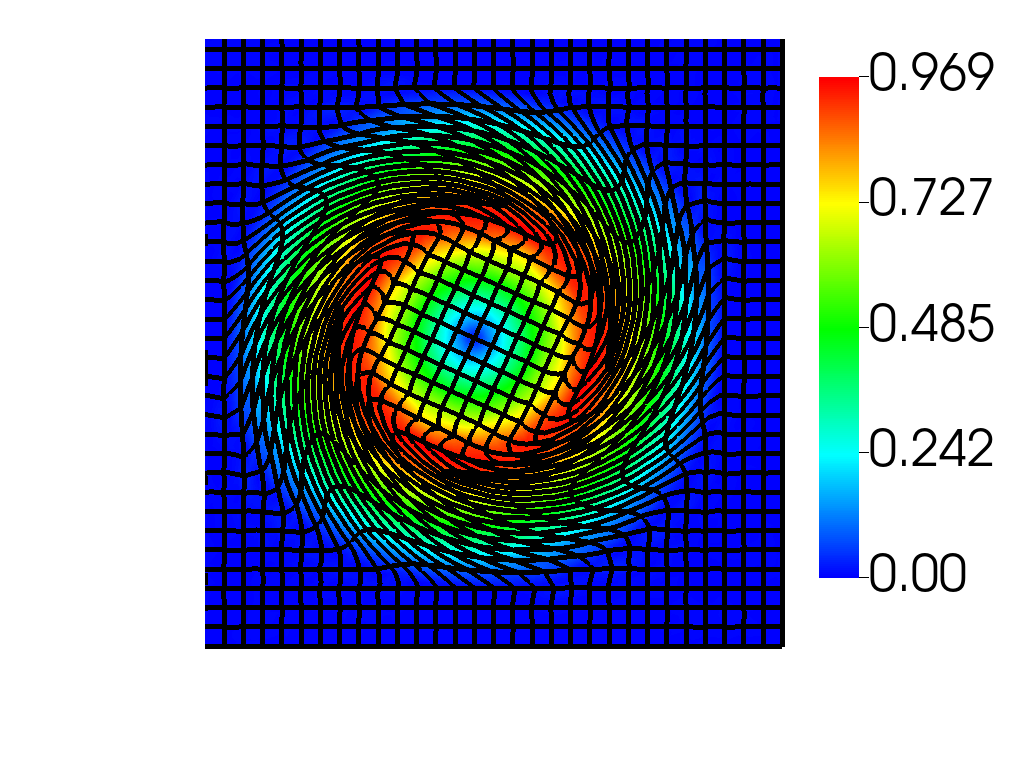}
     \end{subfigure}
     \hfill
     \begin{subfigure}[b]{0.24\textwidth}
         \centering
         \includegraphics[width=\textwidth]{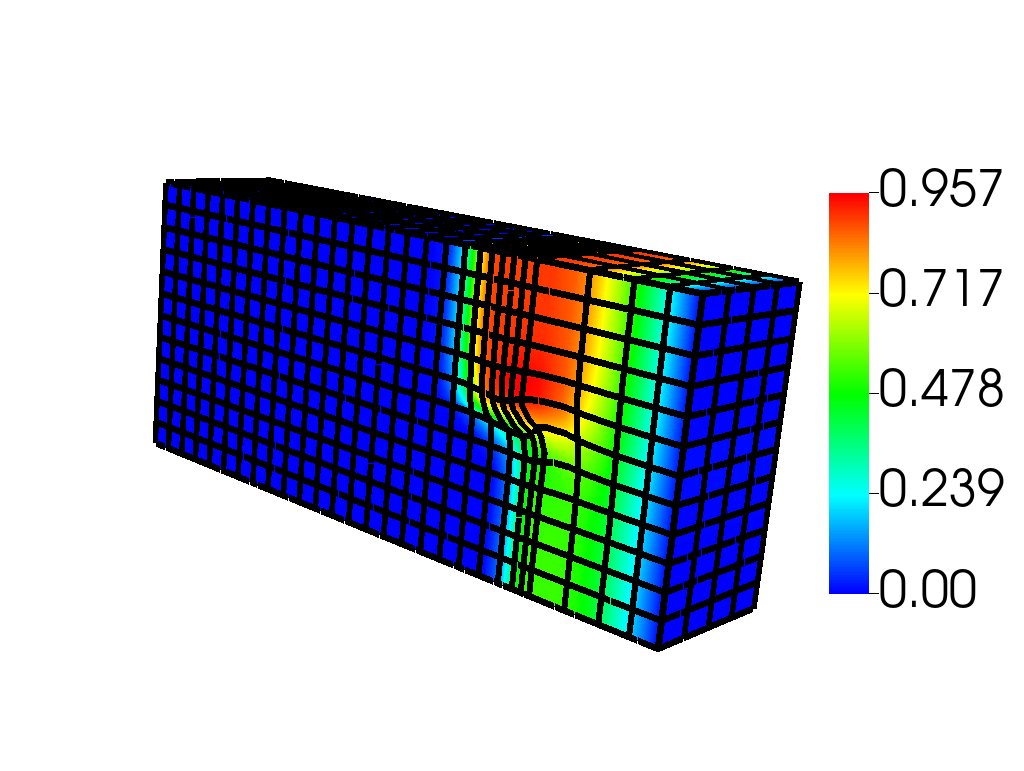}
     \end{subfigure}
     \hfill
     \begin{subfigure}[b]{0.24\textwidth}
         \centering
         \includegraphics[width=\textwidth]{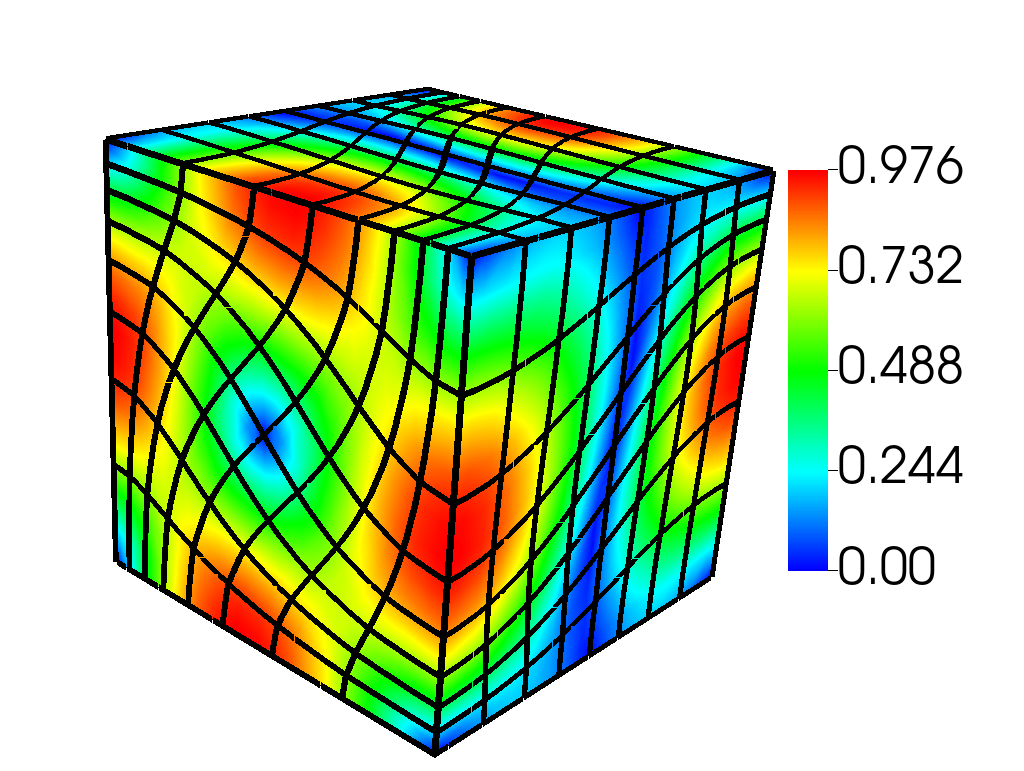}
     \end{subfigure}

     \begin{subfigure}[b]{0.24\textwidth}
         \centering
         \includegraphics[width=\textwidth]{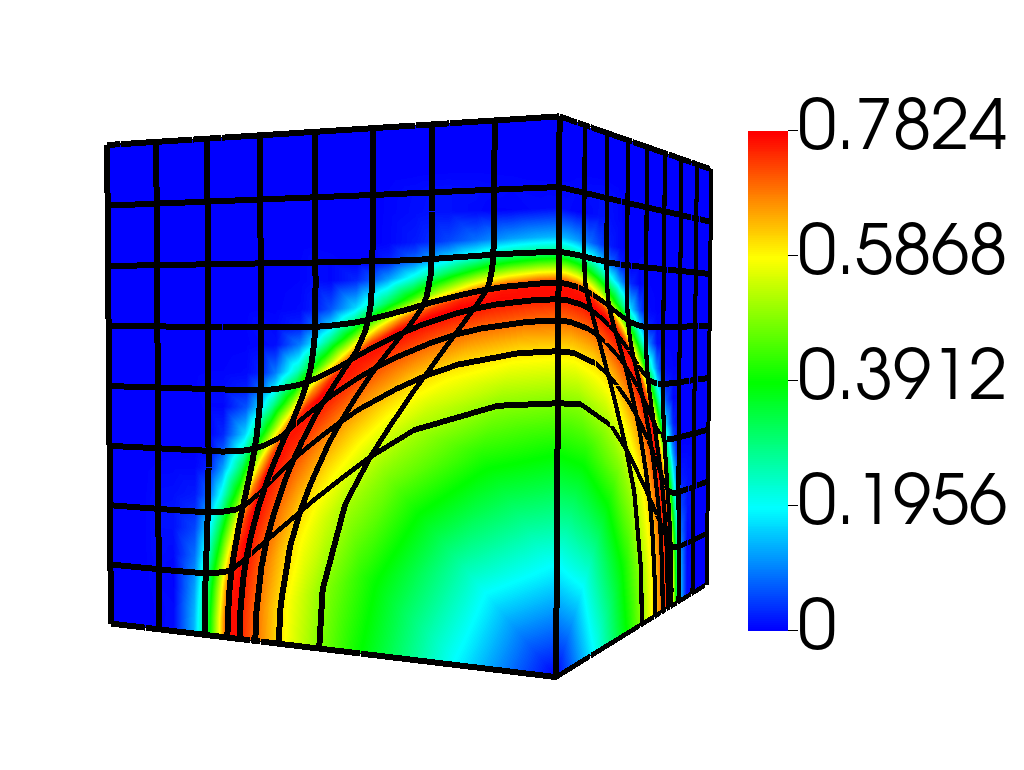}
     \end{subfigure}
     \hfill
     \begin{subfigure}[b]{0.24\textwidth}
         \centering
         \includegraphics[width=\textwidth]{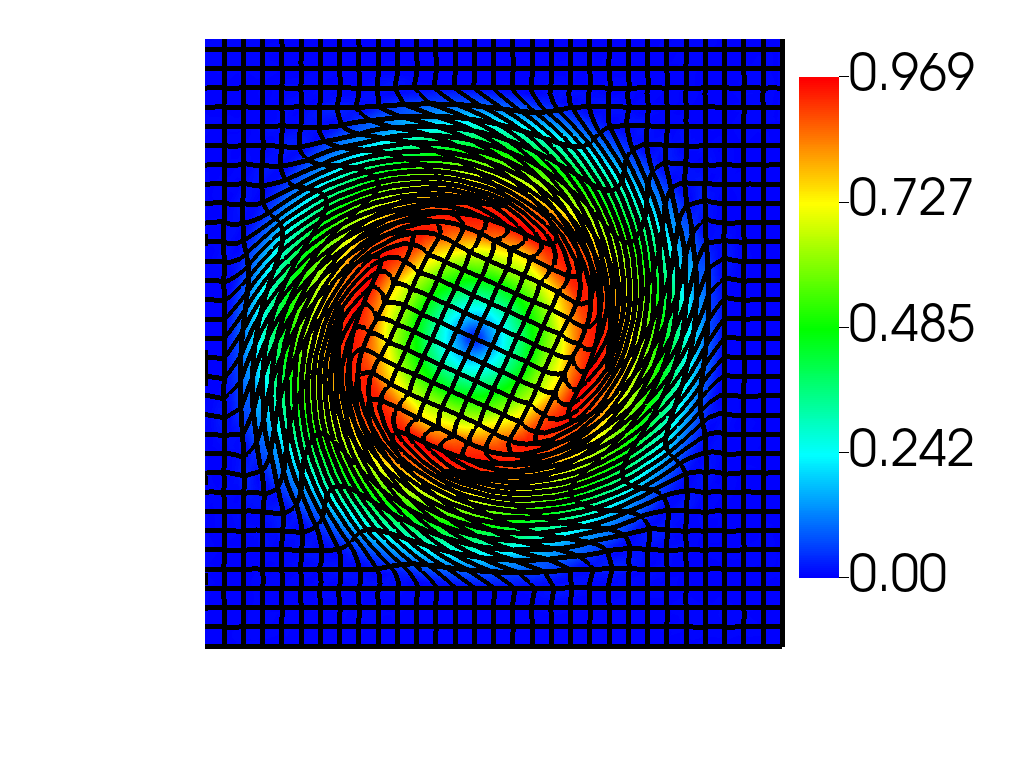}
     \end{subfigure}
     \hfill
     \begin{subfigure}[b]{0.24\textwidth}
         \centering
         \includegraphics[width=\textwidth]{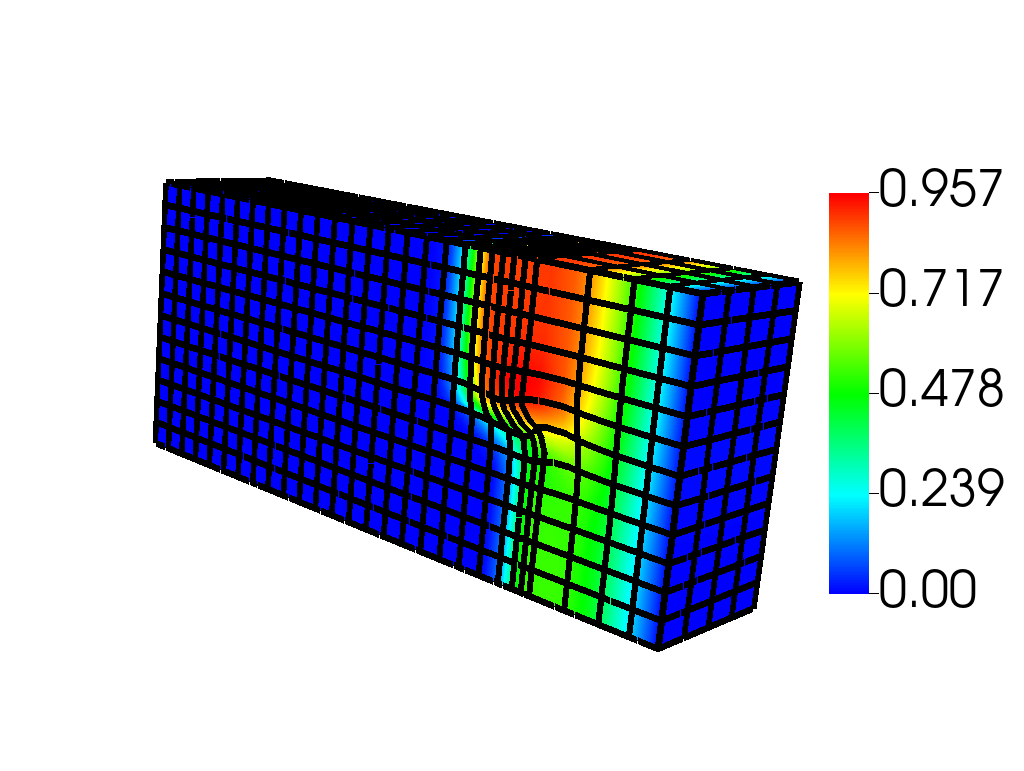}
     \end{subfigure}
     \hfill
     \begin{subfigure}[b]{0.24\textwidth}
         \centering
         \includegraphics[width=\textwidth]{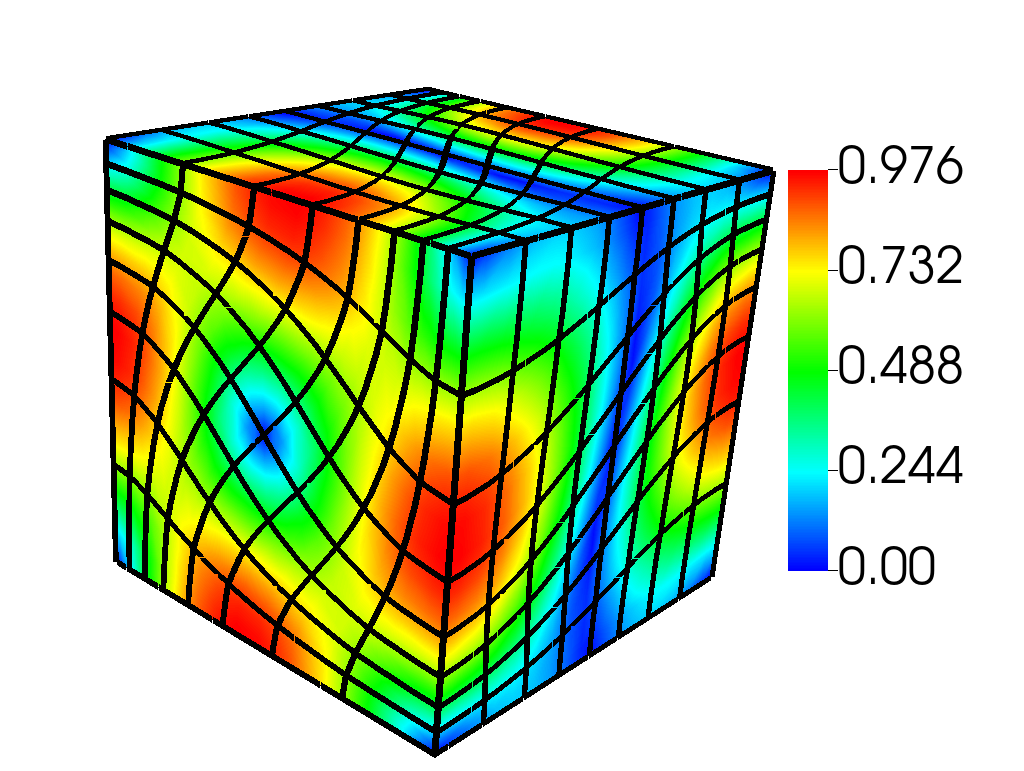}
     \end{subfigure}

     \begin{subfigure}[b]{0.24\textwidth}
         \centering
         \includegraphics[width=\textwidth]{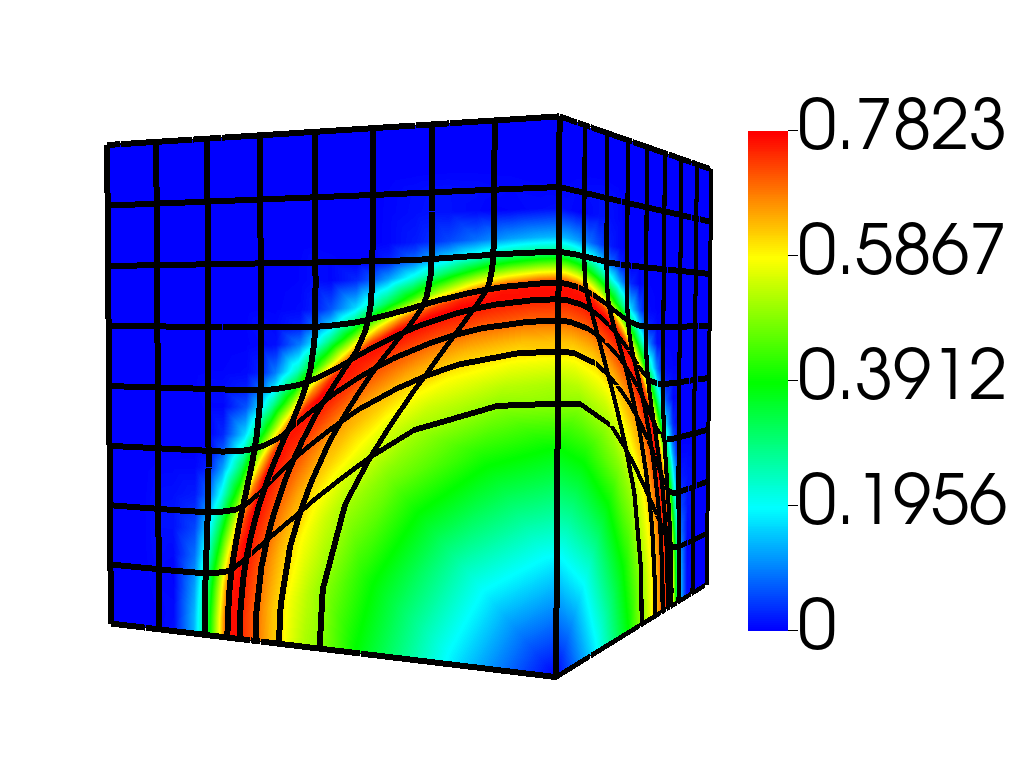}
     \end{subfigure}
     \hfill
     \begin{subfigure}[b]{0.24\textwidth}
         \centering
         \includegraphics[width=\textwidth]{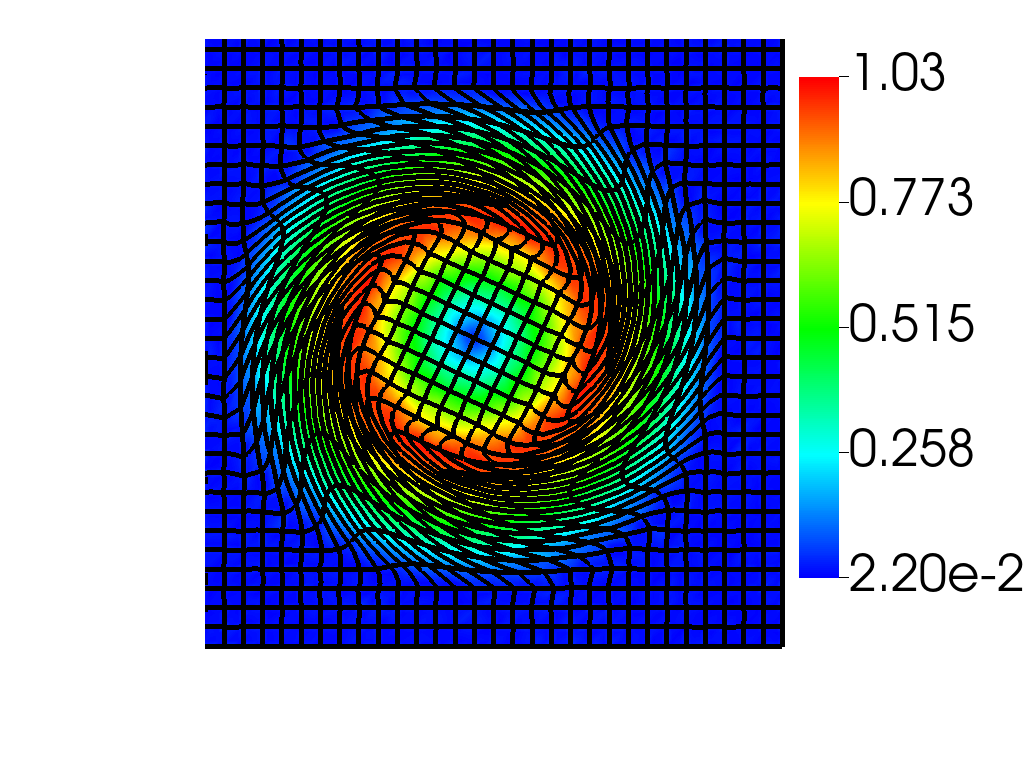}
     \end{subfigure}
     \hfill
     \begin{subfigure}[b]{0.24\textwidth}
         \centering
         \includegraphics[width=\textwidth]{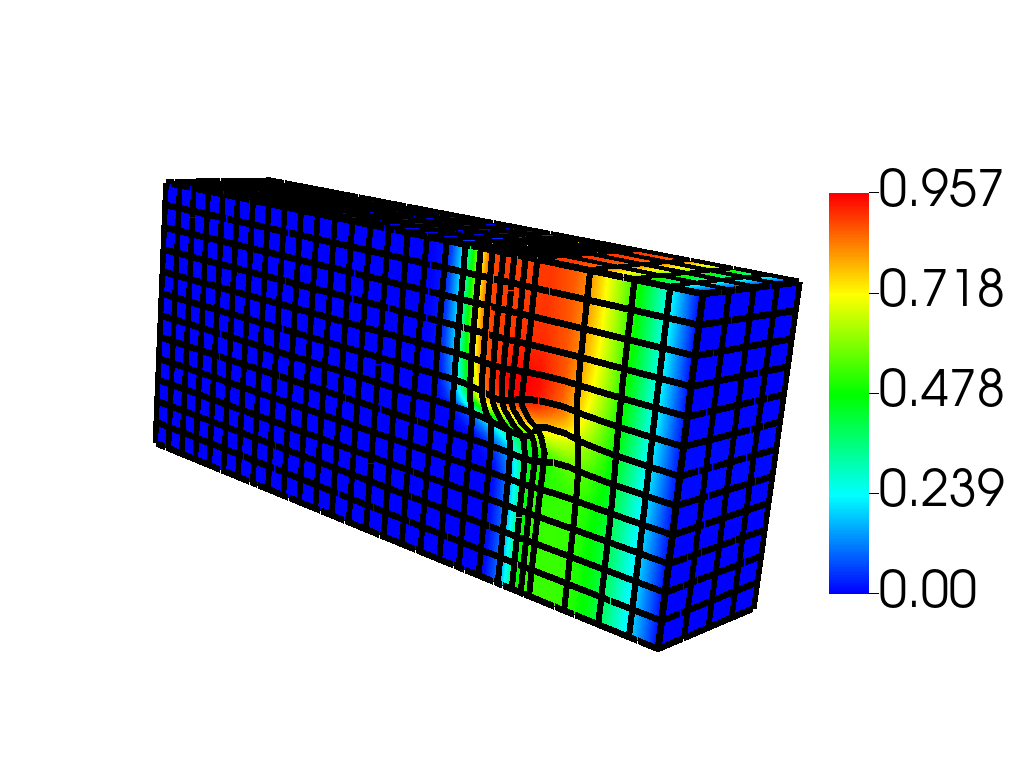}
     \end{subfigure}
     \hfill
     \begin{subfigure}[b]{0.24\textwidth}
         \centering
         \includegraphics[width=\textwidth]{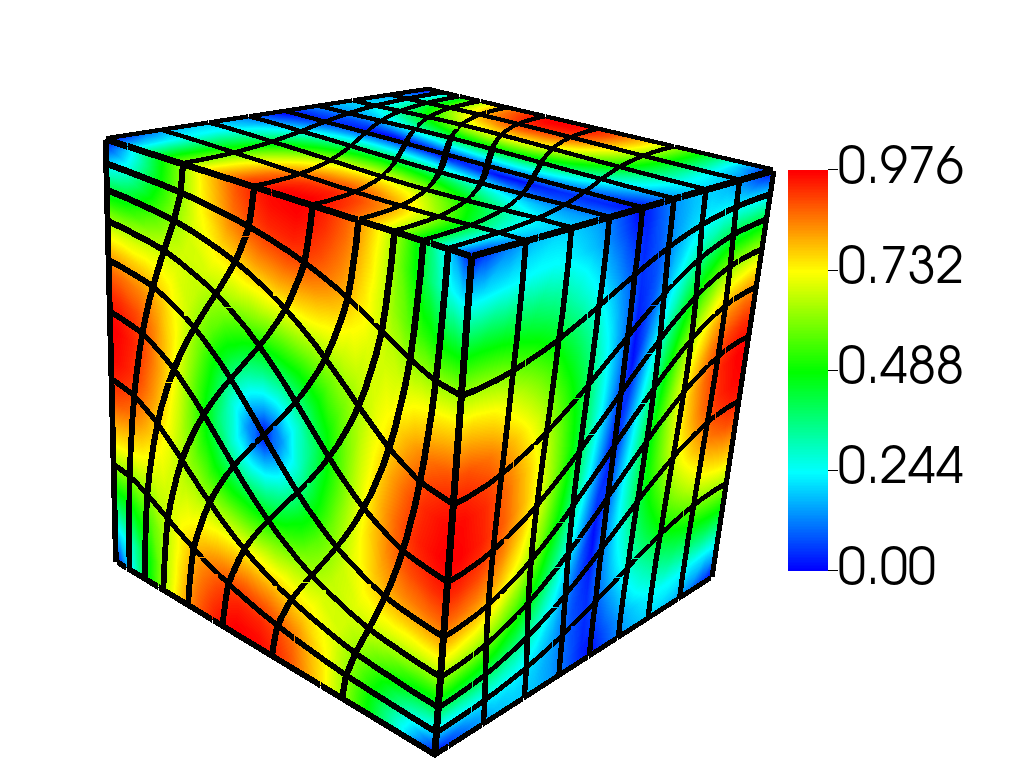}
     \end{subfigure}
        \caption{Final snapshots of the velocity field for each
        simulated problem case.
        The top row corresponds to the results obtained using the
        full model; the middle row using the model hyper-reduced by
        the basic EQP method; the bottom row using the model
        hyper-reduced by the energy conservative EQP method.
        From left to right, the first column corresponds to the
        Sedov problem; the second to Gresho; the third to triple point;
        the fourth to Taylor-Green.}
        \label{fig:results}
\end{figure}

\subsubsection*{Sedov blast}
The Sedov blast problem is a three dimensional generalization of the two
dimensional problem considered in \cite{Dobrev2012}.
The problem simulates the sudden release of energy from a localized source,
leading to the formation of a shock wave propagating in the spatial domain
and converting the initial internal energy into kinetic one.
Its original goal was to test the ability of different artificial viscosity
tensors to maintain radial symmetry of the propagating shock wave on an
irregular discretization mesh.

The computational domain is $\Omega_0=[0,1]^3$.
We use a constant adiabatic index $\gamma=1.4$,
initial density $\rho=1$,
and zero initial velocity.
The initial energy has magnitude $|e|=0.25$ and is all concentrated
at the domain origin.
We employ one of the artificial viscosity stress tensors developed
in \cite{Dobrev2012}
that preserves the propagating shock's radial symmetry.
The simulation time interval is $[0,0.3]$.

\subsubsection*{Gresho vortex}
The Gresho vortex problem is a two dimensional problem originally
posed using the incompressible Euler equations
\cite{Copeland2022}.
Here we use the steady state solution of the original incompressible
problem to build a \emph{manufactured} solution of the compressible
equations.
We do that by substituting the original steady state solution into the
compressible Euler equations and deriving the source terms required to
maintain that steady state \cite{Dobrev2012,Copeland2022}.
The main challenge posed by this problem case is the large deformations
induced on the spatial grid as time evolves.

The computational domain is $\Omega_0=[-0.5,0.5]^2$.
For this problem we use $\gamma=5/3$, initial density $\rho=1$,
initial pressure
\begin{equation*}
p(r,\phi)=\begin{cases}
    5+12.5r^2& \text{if $0\leq r<0.2$}\\
    9-4\log(0.2)+12.5-20r+4\log r& \text{if $0.2\leq r<0.4$}\\
    3+4\log 2& \text{if $r\geq 0.4$}
\end{cases}
\end{equation*}
where $(r,\phi)$ denote the polar coordinates,
and zero artificial viscosity $\sigma_a=0$.
Additionally, we use zero initial radial velocity $v_r=0$ and the initial
angular velocity
\begin{equation*}
v_\phi(r,\phi)=\begin{cases}
    5r& \text{if $0\leq r<0.2$}\\
    2-5r& \text{if $0.2\leq r<0.4$}\\
    0& \text{if $r\geq 0.4$.}
\end{cases}
\end{equation*}
Based on the above, the initial energy is determined by the thermodynamic
equation of state \eqref{eq:EOS}.
The simulation time interval is $[0,0.4]$.

\subsubsection*{Triple point problem}
The triple point problem is a three dimensional Riemann problem
involving two materials in three different states.
The material interaction leads to the formation of shock waves and
generation of vorticity \cite{Dobrev2012,Copeland2022}.

The computational domain is $\Omega_0=[0,7]\times [0,3]\times [0,1.5]$.
We use the adiabatic index
\begin{equation*}
\gamma=\begin{cases}
    1.5& \text{if $x\leq 1$ or $y>1.5$}\\
    1.4& \text{if $x>1$ and $y\leq 1.5$}
\end{cases}
\end{equation*}
zero initial velocity $v=0$, initial density
\begin{equation*}
\rho=\begin{cases}
    1& \text{if $x\leq 1$ or $y\leq 1.5$}\\
    1/8& \text{if $x>1$ and $y>1.5$}
\end{cases}
\end{equation*}
and initial pressure
\begin{equation*}
p=\begin{cases}
    1& \text{if $x\leq 1$}\\
    0.1& \text{if $x>1$.}
\end{cases}
\end{equation*}
The initial energy is computed from the thermodynamic equation of state
\eqref{eq:EOS}.
Given the formation of shock waves, we also use an artificial viscosity
tensor \cite{Dobrev2012}.
The simulated time interval is $[0,0.8]$.

\subsubsection*{Taylor-Green vortex}
The Taylor-Green vortex is a three dimensional generalization of a two
dimensional steady state solution to the incompressible Euler equations
\cite{Dobrev2012,Copeland2022}.
Similarly to the Gresho vortex case, we build a manufactured solution
of the compressible equations by deriving the source terms necessary for
maintaining the original incompressible steady state solution.

The computational domain is $\Omega_0=[0,1]^3$.
We use $\gamma=5/3$, initial density $\rho=1$, initial pressure
\begin{equation*}
    p=100+\frac{[\cos(2\pi x)+\cos(2\pi y)][\cos(2\pi z)+2]-2}{16}
\end{equation*}
and initial velocity
\begin{equation*}
    (v_x,v_y,v_z)=(\sin(\pi x)\cos(\pi y)\cos(\pi z),\,
        -\cos(\pi x)\sin(\pi y)\cos(\pi z),\, 0).
\end{equation*}
The initial energy is computed using the thermodynamic equation of state
\eqref{eq:EOS}.
We use no artificial viscosity stress.
The simulated time interval is $[0,0.4]$.

\begin{table}[t]
    \renewcommand{\arraystretch}{1.1}
    \setlength{\tabcolsep}{10pt}
    \centering
	\begin{tabular}{l l l l l}
	Run & R1 & R2 & R3 & R4\\
	\hline
	Problem & Sedov & Gresho & Triple point & Taylor-Green\\
	$t_f$ & 0.3 & 0.4 & 0.8 & 0.4\\
    $e_{\sigma}$ & 0.9999 & 0.9999 & 0.9999 & 0.9999\\
    $N_s$ & 10 & 10 & 10 & 10\\
    $N_w$ & 91 & 132 & 39 & 418\\
    \hline
	$\epsilon_{x,B}$ & 2.30e-03 & 7.03e-07 & 2.19e-07 & 9.85e-09\\
	$\epsilon_{v,B}$ & 1.37e-02 & 1.19e-05 & 1.46e-05 & 1.41e-05\\
	$\epsilon_{e,B}$ & 2.04e-04 & 2.15e-06 & 4.74e-06 & 2.27e-07\\
	$\Delta E_B$ & 4.23e-04 & 1.72e-08 & 1.15e-06 & 3.03e-11\\
    $\hat{J}_v/J_v$ & $103/32768$ & $47/16384$ & $46/86016$ & $33/32768$\\
    $\hat{J}_e/J_e$ & $61/32768$ & $67/16384$ & $45/86016$ & $48/32768$\\
	$S_B$ & 2.21 & 2.74 & 1.45 & 2.24\\
	\hline
	$\epsilon_{x,C}$ & 4.98e-05 & 3.04e-03 & 4.32e-06 & 2.38e-05\\
	$\epsilon_{v,C}$ & 5.93e-03 & 5.61e-02 & 9.81e-05 & 1.13e-03\\
	$\epsilon_{e,C}$ & 2.34e-03 & 8.87e-03 & 5.43e-05 & 1.36e-05\\
	$\Delta E_{C}$ & 2.28e-13 & 9.28e-15 & 4.43e-14 & 1.38e-14\\
    $\hat{J}/J$ & $179/32768$ & $133/16384$ & $112/86016$ & $104/32768$\\
	$S_{C}$ & 1.68 & 2.60 & 1.38 & 1.69\\
	\hline
	\end{tabular}
    \caption{Parameters and results of the reduced simulations.
    In the table $t_f$ denotes the final time of the simulation;
    $e_\sigma$ the POD energy ratio \eqref{eq:POD-energy};
    $N_s$ the number of samples per window;
    $N_w$ the total number of time windows.
    The symbols $\epsilon_i$ denote the absolute value of the relative
    error of the reduced solution for position ($x$), velocity ($v$)
    and energy ($e$);
    $\Delta E$ the absolute value of the relative energy difference of
    each simulation (start to finish);
    $S$ the speedup of the reduced simulation relative to
    the full one.
    The ratio $\hat{J}/J$ denotes the rounded ratio of quadrature points of
    the reduced rule over that of the full rule, averaged over all windows;
    for basic EQP there are two separate ratios for velocity ($v$)
    and energy ($e$); for conservative EQP there is only one ratio for
    the combined quadrature rule.
    For all results subscript $B$ refers to basic EQP;
    $C$ to energy conservative EQP.}
	\label{tab:results}
\end{table}

\subsection{Results}
Figure \ref{fig:results}
presents the final velocity snapshots obtained using the full
model and the two reduced models for each problem case.
Table \ref{tab:results}
presents the numerical results for the four problem cases
using the two EQP hyper-reduction methods compared in this work.

Starting with the relative energy difference obtained for each
simulated problem, we see that the conservative EQP (CEQP) method
outperforms basic EQP (BEQP) by several orders of magnitude.
The difference in performance is about
9 orders of magnitude for the Sedov case,
6 orders of magnitude for Gresho,
8 orders of magnitude for triple point
and 3 orders of magnitude for Taylor-Green.
In addition, the energy difference achieved with CEQP is
near machine precision for all four simulated problems.
These results prove that the changes made in the derivation of the
reduced quadrature rule do indeed lead to high accuracy in the
conservation of energy, as expected from the theoretical analysis.

In terms of the differences in relative errors achieved for the position,
velocity and energy solutions, the two methods perform similarly
in Sedov and triple point but demonstrate more pronounced
discrepancies in Gresho and Taylor-Green.
More specifically, for the Sedov case both methods lead to
similar errors, with CEQP outperforming BEQP in position and
velocity and conversely in energy.
Similarly, the two methods yield similar results for triple point,
with BEQP performing better by about an order of magnitude in all
three variables.
The situation is different for Taylor-Green and Gresho.
For Taylor-Green, CEQP underperforms by about 2-3 orders of magnitude
in the three solution errors,
while for Gresho by about 3-4 orders of magnitude.
Despite the more pronounced differences in accuracy in these
problem cases, the errors obtained by CEQP are still
reasonably low.

Finally, we arrive at the comparison of the simulation speedup
(acceleration) values obtained with each EQP method.
We see that the two methods yield similar speedup results,
although BEQP is outperforming CEQP in all four problem cases.
The moderate differences in speedup demonstrate that the
changes introduced into the derivation of the
reduced quadrature rule for CEQP has not hurt that aspect of
the method's performance.

\subsection{Discussion}\label{sec:discussion}
Among the presented numerical results, the discrepancy in solution
accuracy between the basic and energy conservative EQP methods is
largest in the Gresho problem case.
In addition, Gresho is the problem case generating the largest values
of solution errors for CEQP, despite the fact that its energy
conservation performance is on par with the other experiments.
As of this writing, it is not well understood why that is the case,
and whether the elevated errors are due to the numerical implementation
of the method or due to some intrinsic challenge posed by the
Gresho problem.

The reported speedup factors are relatively small compared to those
reported in \cite{Copeland2022},
where interpolation methods are used for hyper-reduction.
This is partly because our results are generated using a
numerical implementation of the basic and energy conservative
EQP methods that is not fully optimized yet.
More specifically, the used implementation performs some of the
lifting of the reduced state vectors to full state vectors using
all degrees of freedom, instead of using only the
degrees of freedom sampled by the reduced quadrature rules.
Independently of those optimizations,
we also expect that the speedup values will increase with
increasing values for the mesh refinement parameter $m$
and the FEM basis polynomial order $k$.

A direction of future research is comparing the performance
of different NNLS algorithms, especially as it relates to the
sparsity and accuracy of the derived reduced quadrature rules.
As shown in Table \ref{tab:results}, the Lawson-Hanson algorithm
used in this work leads to sparse quadrature rules.
Nevertheless, it is not clear how different methods
would compare against our reported results.
In addition, we intend to investigate why
the $LQ$ preconditioning of the NNLS constraints matrix
(ensuring orthogonality of constraints)
leads to improved performance of the Lawson-Hanson algorithm.
Although observed empirically, to our knowledge it is not
fully understood why that happens
\cite{Humphry2025}.

In the present work we focus on reproductive simulations,
where each reduced model is tested using initial conditions that
have been used in its training.
The developed methods readily generalize to predictive simulations,
where out of sample initial conditions are used to test the model's
performance.
We intend to pursue this in the near future.

Conceptually, the model reduction approach employed in this work
builds upon the finite element framework.
Starting with the FEM discretization of the problem, with the
FEM basis functions acting as the trial and test functions,
we proceed to replace these functions by ones derived from
simulation data.
One important difference between the two model formulations is that
the FEM model employs trial and test functions that are local in
space, meaning that their support is confined within a small number
of elements.
On the contrary, the reduced models employ trial and/or test
functions with a support that is generally extended across
a large number of elements, possibly including the whole
spatial domain.
Although this allows us to use a smaller number of degrees of freedom
to represent the considered dynamics, it can also lead to some of the
extrapolation problems encountered by reduction methods to date,
since the accurate representation and prediction of out of sample
dynamics cannot be guaranteed.

There is ongoing work related to deriving data based reduction methods
that use spatially local reduced basis functions, bringing them
closer to a data based analog of the finite element framework
\cite{Eftang2013,Phuong2013,McBane2021,Chung2024,Choi2025}.
It is interesting to explore the applicability of such methods to
the problem setting considered in the present work, especially since
quadrature hyper-reduction methods are naturally well suited to the
finite element framework.

\section{Conclusion}\label{sec:conclusion}
We developed reduced models for nonlinear Lagrangian hydrodynamics
problems using quadrature hyper-reduction to approximate the
evaluation of the nonlinear terms with reduced computational cost.
The hyper-reduction was performed using the EQP method,
which is well suited to the finite element discretization of
the considered nonlinear conservation laws.

Starting with an application of the basic EQP method, we developed
a variant of the method that preserves the energy conservation of
the full model in the reduction process.
In addition to proving exact energy conservation in
the resulting reduced model, we performed numerical simulations of
four different benchmark problems demonstrating the conservation of
energy to near machine precision.
Our reported results show that the additional benefit of energy
conservation is obtained with only a moderate reduction in the
accuracy or speedup of the reduced simulations compared to the
basic EQP method.

\section*{Acknowledgments}
This work was performed under the auspices of the
U.S. Department of Energy (DOE)
by Lawrence Livermore National Laboratory (LLNL)
under Contract No. DE-AC52–07NA27344.
Part of this work was performed while C. Vales was
a research intern at LLNL supported by the
Mathematical Sciences Graduate Internship (MSGI)
program administered by the
U.S. National Science Foundation.
C. Vales also acknowledges support from the
U.S. Department of Energy under grant DE-SC0025101.
Y. Choi was supported for this work by the
U.S. Department of Energy, Office of Science,
Office of Advanced Scientific Computing Research,
as part of the CHaRMNET Mathematical Multifaceted Integrated
Capability Center (MMICC) program,
under Award Number DE-SC0023164.
IM release: LLNL-JRNL-2010206.

\appendix
\section{Implementation details}\label{app:implementation}
In this section we offer additional details on the
implementation of the basic and conservative EQP methods
developed in the main text.

\subsection{Basic EQP}\label{app:beqp}
We consider the nonnegative least squares (NNLS) formulation
used for the basic EQP method developed in Section \ref{sec:eqp},
as well as the preconditioning operations performed
before solving the formed NNLS problem.

\subsubsection*{NNLS formulation}
To solve the linear optimization problem for the reduced
quadrature rule described in Section \ref{sec:eqp},
we reformulate it as an NNLS problem for the vector of quadrature
weights $\tilde{\vect{\rho}}^v\in\Rbb^{J_v}$
\cite{Du2022,Sleeman2022}.
Given the known vector of full quadrature weights
$\vect{\rho}^v\in\Rbb^{J_v}$,
we form the accuracy constraints matrix
$\matr{C}_v\in\Rbb^{N_c^v\times J_v}$, $N_c^v\ll J_v$,
with entries
\begin{equation*}
    C_{v,sj}=g_{v,i}^\psi(x^v_j,\tilde{w}(t_k)),\quad
    0\leq i<n_v,\quad
    0\leq k<N_t,\quad
    0\leq j<J_v
\end{equation*}
with row index $0\leq s=i+kn_v<N_c^v$.
Denoting $\vect{b}=\matr{C}_v\vect{\rho}^v\in\Rbb^{N_c^v}$,
the vector $\tilde{\vect{\rho}}^v$
is found by solving the NNLS problem
\begin{equation}\label{eq:nnls-problem}
    \tilde{\vect{\rho}}^v=\argmin_{\vect{r}\geq\vect{0}}
        \lVert\matr{C}_v\vect{r}-\vect{b}\rVert_2^2.
\end{equation}
More specifically, for a given vector of error thresholds
$\vect{\epsilon}\in\Rbb^{N_c}$,
we use the iterative Lawson-Hanson algorithm to solve the following
problem for $\vect{r}\geq\vect{0}$
\begin{equation}\label{eq:nnls-problem2}
    |\vect{c}_v^s\cdot\vect{r}-b_s|\leq\epsilon_s,\quad
        0\leq s<N_c
\end{equation}
where $\vect{c}_v^s\in\Rbb^{J_v}$ denotes the $s$th row
of matrix $\matr{C}_v$,
and $b_s$, $\epsilon_s$ the $s$th respective entries of vectors
$\vect{b}$ and $\vect{\epsilon}$
\cite{Lawson1995,Chapman2017}.

In the NNLS formulation, \eqref{eq:nnls-problem} is used to
enforce the accuracy constraints of the linear optimization problem.
The nonnegativity constraints are enforced implicitly by the employed
algorithm.
Finally, compared to the original form of the problem, the NNLS
formulation does not explicitly promote sparsity by minimizing
the 1-norm of the solution vector $\tilde{\vect{\rho}}^v$.
Nevertheless, several numerical experiments---including the ones
we report in this work---show that sparse solutions are obtained
by the Lawson-Hanson algorithm
\cite{Du2022,Sleeman2022,Larsson2026}.

\subsubsection*{Preconditioning}
Before solving \eqref{eq:nnls-problem},
we perform two conditioning operations on the constraints matrix
$\matr{C}_v$, both to reduce the number of iterations required by
the Lawson-Hanson algorithm to converge and to increase the derived
solution's accuracy \cite{Humphry2025}.
First, we rescale each row of $\matr{C}_v$
so that its entries have absolute value less than or equal to one.
Second, we use the reduced LQ decomposition of
$\matr{C}_v=\matr{L}_v\matr{Q}_v$,
where $\matr{L}_v\in\Rbb^{N_c^v\times N_c^v}$ is lower triangular
and $\matr{Q}_v\in\Rbb^{N_c^v\times J_v}$ has orthonormal rows,
to transform the NNLS problem \eqref{eq:nnls-problem} to
\begin{equation}\label{eq:nnls-lq-problem}
    \tilde{\vect{\rho}}^v=\argmin_{\vect{r}\geq \vect{0}}
        \bigl\|\matr{Q}_v\vect{r}-\tilde{\vect{b}}\bigr\|_2^2
\end{equation}
where $\tilde{\vect{b}}=\matr{Q}_v\vect{\rho}^v\in\Rbb^{N_c^v}$,
under the assumption that $\matr{L}_v$ is nonsingular.
In addition, we modify the error thresholds
$\epsilon_s$ accordingly to ensure that the solution obtained
from \eqref{eq:nnls-lq-problem}
will also satisfy
\eqref{eq:nnls-problem2}
with the original error parameters.
We define the new error thresholds $\tilde{\epsilon}_s$ by
\begin{equation}\label{eq:nnls-lq-error}
    \tilde{\epsilon}_s=\min_{s\leq j<N_c}
        \frac{1}{(j+1)\vert L_{js}\vert}\,\epsilon_j,\quad
        0\leq s<N_c
\end{equation}
which is sufficient to ensure that the transformed problem's
solution satisfies \eqref{eq:nnls-problem2} \cite{Humphry2025}.
The transformed NNLS problem reads
\begin{equation}\label{eq:nnls-lq-problem2}
    |\vect{q}_v^s\cdot\vect{r}-\tilde{b}_s|
        \leq\tilde{\epsilon}_s,\quad 0\leq s<N_c
\end{equation}
where
$\vect{q}_v^s\in\Rbb^{J_v}$ denotes the $s$th row of
matrix $\matr{Q}_v$.

\subsection{Conservative EQP}
We consider the implementation of the NNLS problem for the energy
conservative EQP method developed in Section \ref{sec:ceqp},
as well as an additional basis enrichment operation
performed to ensure conservation of the discrete total
energy for simulations consisting of multiple time windows.

\subsubsection*{NNLS formulation}
Following the approach outlined in Section \ref{sec:eqp}
and \ref{app:beqp},
the reduced quadrature rule is computed by solving an NNLS problem
with a total of $N_c=(n_v+n_e)N_t$ accuracy constraints,
where $N_t$ denotes the number of snapshots of $\tilde{w}$
used to form the constraints.
The constraints are used to control the accuracy of a
single reduced quadrature rule that will be used for the
approximation of both force vectors
$\vect{F}_v^\phi$ and $\vect{F}_e^\phi$
introduced in \eqref{eq:rom-force-ceqp}
\begin{equation*}
    \vect{F}_v^\phi(\tilde{\vect{w}})=
        \int_{\Omega(t)}\vect{g}_v^\phi(x,\tilde{w})dx\qquad
    \vect{F}_e^\phi(\tilde{\vect{w}},\tilde{\vect{v}})=
        \int_{\Omega(t)}\vect{g}_e^\phi(x,\tilde{w},\tilde{v})dx
\end{equation*}
where the respective integrands
$\vect{g}_v^\phi$ and $\vect{g}_e^\phi$
take values in $\Rbb^{n_v}$ and $\Rbb^{n_e}$
and have components
\begin{equation*}
    g_{v,i}^\phi(x,\tilde{w})=
        \sigma(\tilde{w})(x):\nabla\phi^v_i(x)\qquad
    g_{e,i}^\phi(x,\tilde{w},\tilde{v})=
        (\sigma(\tilde{w})(x):\nabla\tilde{v}(x))\,\phi^e_i(x).
\end{equation*}
The accuracy constraints are encoded in the matrix
$\matr{C}\in\Rbb^{N_c\times J}$,
where the first $n_vN_t$ rows are used to enforce the
constraints related to $\vect{F}_v^\phi$
\begin{equation*}
    C_{sj}=g_{v,i}^\phi(x_j,\tilde{w}(t_k)),\quad
    0\leq i<n_v,\quad 0\leq k<N_t,\quad 0\leq j<J
\end{equation*}
with constraint index $s=i+kn_v$,
and the next $n_eN_t$ rows the constraints
related to $\vect{F}_e^\phi$
\begin{equation*}
    C_{sj}=g_{e,i}^\phi(x_j,\tilde{w}(t_k),\tilde{v}(t_k)),\quad
    0\leq i<n_e,\quad 0\leq k<N_t,\quad 0\leq j<J
\end{equation*}
with constraint index $s=n_vN_t+i+kn_e$.
In the above, $\tilde{w}(t_k)$
denotes the snapshots used to form the constraints.

\subsubsection*{Reduced bases enrichment}
When the simulation consists of more than one time window,
an additional modification is made to the reduced bases
$\matr{\Phi}_v$ and $\matr{\Phi}_e$
to ensure that the total energy is conserved across windows.
Let $\tilde{\vect{v}}^{(n-1)}$ and $\tilde{\vect{e}}^{(n-1)}$
denote the velocity and energy solution vectors after the final
timestep of window $n-1$, $n>1$, has been completed
\begin{equation*}
    \tilde{\vect{v}}^{(n-1)}=
        \matr{\Phi}_v^{(n-1)}\hat{\vect{v}}^{(n-1)}\qquad
    \tilde{\vect{e}}^{(n-1)}=
        \matr{\Phi}_e^{(n-1)}\hat{\vect{e}}^{(n-1)}
\end{equation*}
where superscripts are used to mark the time window index.
Accordingly, we denote by
$\tilde{\vect{v}}^{(n)}$ and $\tilde{\vect{e}}^{(n)}$
the corresponding vectors after the change from window $n-1$ to
window $n$ has been performed
\begin{align*}
    \tilde{\vect{v}}^{(n)}&=\matr{\Phi}_v^{(n)}\hat{\vect{v}}^{(n)}=
        \matr{\Phi}_v^{(n)}\matr{\Phi}_v^{(n)\top}
        \matr{\Phi}_v^{(n-1)}\hat{\vect{v}}^{(n-1)}\\
    \tilde{\vect{e}}^{(n)}&=\matr{\Phi}_e^{(n)}\hat{\vect{e}}^{(n)}=
        \matr{\Phi}_e^{(n)}\matr{\Phi}_e^{(n)\top}
        \matr{\Phi}_e^{(n-1)}\hat{\vect{e}}^{(n-1)}.
\end{align*}
Note that no timestepping has been performed after switching from
window $n-1$ to $n$, so the two pairs of vectors should be identical.
To ensure that this is the case, we extend the basis matrices
$\matr{\Phi}_v^{(n)}$ and $\matr{\Phi}_e^{(n)}$
by adding the respective solution vectors
$\tilde{\vect{v}}^{(n-1)}$ and $\tilde{\vect{e}}^{(n-1)}$
as an additional column and reorthonormalize.
This guarantees that 
$\tilde{\vect{v}}^{(n-1)}$ and $\tilde{\vect{e}}^{(n-1)}$
are in the span of the respective basis matrices
$\matr{\Phi}_v^{(n)}$ and $\matr{\Phi}_e^{(n)}$.
In turn, this ensures that energy is conserved exactly when
changing windows
\begin{equation*}
    TE(\tilde{\vect{w}}^{(n)})-TE(\tilde{\vect{w}}^{(n-1)})
        =\vect{1}_\Ecal^{\top}\matr{M}_e
        (\tilde{\vect{e}}^{(n)}-\tilde{\vect{e}}^{(n-1)})
        +\frac{1}{2}(\tilde{\vect{v}}^{(n)\top}
        \matr{M}_v\tilde{\vect{v}}^{(n)}-\tilde{\vect{v}}^{(n-1)\top}
        \matr{M}_v\tilde{\vect{v}}^{(n-1)})=0.
\end{equation*}
No such extension is required for the first window ($n=0$).
Since the extension of the basis matrices
$\matr{\Phi}_v^{(n)}$ and $\matr{\Phi}_e^{(n)}$
requires knowledge of the state vectors
$\tilde{\vect{v}}^{(n-1)}$ and $\tilde{\vect{e}}^{(n-1)}$,
this operation must be performed during the online stage at the time
of the window change.
To extend and incrementally reorthonormalize the basis matrices
efficiently, we employ the modified Gram-Schmidt algorithm with
double orthogonalization \cite{Giraud2005}.

\bibliographystyle{elsarticle-num}
\biboptions{sort&compress}
\bibliography{refs}

\end{document}